\documentclass[11pt, reqno]{amsart}

\usepackage{color}
\usepackage{amsmath}
\usepackage{latexsym,amssymb}
\usepackage{graphicx}
\usepackage{stmaryrd}

\usepackage{mathptmx}
\usepackage{enumerate}

\newtheorem{thm}{ \bf Theorem}[section]

\newtheorem{lem}[thm]{ \bf Lemma}
\newtheorem{prop}[thm]{ \bf Proposition}

\newtheorem{rem}[thm]{ \bf Remark}

\numberwithin{equation}{section}

\newcommand{\noi}{\noindent}

\newcommand{\R}{\mathbb{R}}
\newcommand{\N}{\mathbb{N}}

\newcommand{\la}{\lambda}
\newcommand{\sig}{\sigma}

\newcommand{\eps}{\varepsilon}

\newcommand{\del}{\delta}
\newcommand{\om}{\omega}

\newcommand{\Gam}{\mathnormal{\Gamma}}
\newcommand{\Del}{\mathnormal{\Delta}}

\newcommand{\Sig}{\mathnormal{\Sigma}}

\newcommand{\Ps}{\mathnormal{\Psi}}

\newcommand{\Om}{\mathnormal{\Omega}}

\newcommand{\C}{{\mathbb C}}
\newcommand{\D}{{\mathbb D}}

\newcommand{\Z}{{\mathbb Z}}

\newcommand{\PP}{{\mathbb P}}

\newcommand{\calS}{{\mathcal{S}}}

\newcommand{\be}{{\mathbf e}}

\newcommand{\SC}{\mathfrak{S}}

\DeclareMathOperator*{\argmax}{arg\,max}

\newcommand{\skp}{\vspace{\baselineskip}}

\newcommand{\diag}{{\rm diag}}

\newcommand{\w}{\wedge}

\newcommand{\To}{\Rightarrow}

\newcommand{\wh}{\widehat}

\newcommand{\iy}{\infty}

\newcommand{\IA}{\text{\it IA}}
\newcommand{\AT}{\text{\it AT}}
\newcommand{\RT}{\text{\it RT}}
\newcommand{\JT}{\text{\it JT}}
\newcommand{\WT}{\text{\it WT}}
\newcommand{\mumin}{\mu_{\text{min}}}

\begin{document}

\baselineskip=17pt

\title[ ]
{An $\eps$-Nash equilibrium with high probability
for strategic customers in heavy traffic}

\author[Rami Atar]{Rami Atar}
\address{Department of Electrical Engineering\\
Technion -- Israel Institute of Technology\\
Haifa 32000, Israel}
\email{atar@ee.technion.ac.il}

\author[Subhamay Saha]{Subhamay Saha}
\address{Department of Electrical Engineering\\
Technion -- Israel Institute of Technology\\
Haifa 32000, Israel}
\email{subhamay@tx.technion.ac.il}

\thanks{Research supported in part by the ISF (Grant 1315/12)}


\date{}

\begin{abstract}
A multiclass queue with many servers is considered, where customers
make a join-or-leave decision upon arrival based on queue length
information, without knowing the scheduling policy or the state of other queues.
A game theoretic formulation is proposed and analyzed, that takes advantage
of a phenomenon unique to heavy traffic regimes, namely Reiman's snaphshot principle,
by which waiting times are predicted with high precision by the information
available upon arrival. The payoff considered
is given as a {\it random variable}, which depends on the customer's decision,
accounting for waiting time in the queue and penalty for leaving.
The notion of an equilibrium is only meaningful in an asymptotic framework,
which is taken here to be the Halfin-Whitt heavy traffic regime.
The main result is the identification
of an $\eps$-Nash equilibrium with probability approaching 1.
On way to proving this result, new diffusion limit results
for systems with finite buffers are obtained.

\skp

\noi

\skp

\noi

\end{abstract}

\maketitle

\section{\textbf{Introduction}}\label{sec1}

Equilibrium behavior of strategic customers
in queueing systems has been the subject of great interest since the
work of Naor \cite{naor} (see the book by Hassin and Haviv \cite{hashav} for a survey),
and has been a particularly active research area in recent years.
As far as heavy traffic analysis is concerned,
not a great deal of attention has been drawn to game theoretic aspects
such as the asymptotic study of Nash equilibria, unlike, for example,
control theoretic treatment, to which much work has been devoted.
In this paper we propose and analyze a game theoretic formulation of strategic
customers in a multi-class queueing system,
that takes advantage of phenomena specific to heavy traffic regimes.
The formulation is based on associating with each customer a payoff
that reflects the customer's actual waiting time rather than its expectation.
The notion of equilibrium addressed, namely an $\eps$-Nash equilibrium with high probability (w.h.p.),
becomes meaningful only as scaling limits are taken.
An additional aspect that is unique to this setting regards the relatively
small level of information required for the players.
In game theoretic analysis of queueing models, it is usually the case that
when partial information of the system's state is available to the player,
the unobservable states are assumed to be in stationarity.
In the setting of this paper,
customers are aware of the queue length of their own class but not those of other classes,
and moreover, the scheduling policy is not known to them.
However, stationarity assumptions are not required.

The model considered consists of a fixed number of customer classes, that
differ in their service rates,
and $n$ identical, exponential servers that work in parallel.
Upon arrival of a class-$i$ customer, the $i$th queue length
is revealed and, based on this information,
he decides whether to join or leave. Accordingly, the customer's payoff is given by
$h_i(\WT)$ or $r_i$, respectively, where $h_i:\R_+\to\R_+$ is a given
function, $\WT$ is the time the particular customer will wait in line before
being admitted into service, and $r_i\in(0,\iy)$ is a
cost for not receiving service (both $h_i$ and $r_i$ depend on the class, $i$).
Because $\WT$ is a random variable the value of which is not known at the
time of decision, the payoff is in fact a {\it random} function of the customer's
decision, as well as other customers' decisions.
Establishing an equilibrium based on random payoffs is made possible
thanks to the consideration of the game in an asymptotic regime.

The asymptotic setting considered is the Halfin-Whitt (HW) heavy traffic regime \cite{hal-whi},
in which the number of servers, $n$, grows without bound, and the arrival processes
accelerate accordingly so as to keep the system critically loaded.
The customers considered are those that arrive during a fixed, finite time interval.
Thus the number of participating players also grows without bound.

The specific feature due to which a random payoff formulation is tractable
in this regime (and potentially in other heavy traffic regimes)
is {\it Reiman's snapshot principle} (RSP)
\cite{rei-snap}, which, when specialized to the present setting,
states that the waiting time a customer will experience
is asymptotically equal to the queue length at the time of arrival divided by the
overall rate at which customers from the class are served (see Section \ref{sec5}
for a precise statement). While
this principle has been proved in a number of settings, it does not always hold
(as explained in Remark \ref{rem1} below). In particular, its validity depends
on the scheduling policy. Our equilibrium results, that are based on this principle,
can therefore only be obtained under some assumptions on the scheduling.
We address this aspect by considering two families of scheduling policies under which,
as we show, RSP holds:
{\it fixed priority} (FP), where a server that becomes available will
always pick the customer at the head of the line of the buffer with least index
among non-empty buffers, and {\it serve the longest queue} (SLQ),
where the buffer with longest queue is picked.
Our main result shows that if all customers adopt a strategy that uses RSP
as a prediction for the waiting time, an $\eps$-Nash equilibrium w.h.p.\ is obtained.

On way to proving the main result we prove new diffusion limit results for the above
two policies, for systems in which customers join only when the queue length
of the corresponding buffer is below a threshold, an element
that can otherwise be described by finite buffers. A non-standard aspect
of the diffusion scale analysis required toward proving the main result
is that one must take into account different behaviors of customers, so as to allow for
scenarios where one of the customers deviates from the strategy that is to be shown
to lead to an equilibrium. In particular, properties on which the proof is based, such as
the $C$-tightness of some of the processes involved, are proved to hold uniformly
over such scenarios.

The only work the authors are aware of where a game theoretic equilibrium
is considered in conjunction with heavy traffic analysis of a queueing model
is by Gopalakrishnan et al.\ \cite{ward14}, that studies
{\it servers} that act strategically. Specifically, they
choose their service rate in order to optimize a tradeoff between
an effort cost and value of idleness. The focus of \cite{ward14} is on
the study of the implications of such strategic behavior on staffing
and routing, as the size of the system becomes large.
The notion of equilibrium is not of the type considered here, but is
based on deterministic, steady state payoffs, as well as on
complete state information. Moreover, the asymptotic analysis is provided {\it subsequently}
to establishing the prelimit equilibrium.

As far as our convergence results and RSP
are concerned, the closest work is by Gurvich and Whitt \cite{gur-whi},
where a parallel server system, with multiple classes as well as multiple server
pools, is considered in the HW regime, under the
{\it fixed queue and idleness ratio} policy. This policy aims at keeping queue lengths
as well as idleness levels at the different server pools at predetermined
fixed ratios. When specialized to the case of a single server pool,
and equal queue length ratios, this setting is similar to one of the two
settings studied in this paper, namely SLQ.
There are, however, two important differences in terms of the technical treatment.
First, as already mentioned,
the estimates required to deduce the main result must be uniform over scenarios.
A second difference is that finite buffers are not covered by \cite{gur-whi}.
Although it may seem that this aspect requires only simple
adaptations to cover convergence results,
this is not the case. In fact, diffusion limits do
not always exist under our assumptions, as is the case under SLQ if
the buffers are of equal size (this issue is developed further in \cite{ata-sah-2}).
Hence considerations beyond the infinite buffer model are necessarily significant here.

As an additional small sample of recent work on strategic behavior in queueing systems,
we mention Guo and Hassin \cite{guo}, that analyze the response of customers
to shutting down service
when the queue is empty, and resuming when the queue length exceeds a threshold; and
Manou et al.\ \cite{manou}, that studies a natural model for
the behavior of customers in a transportation station.
In both cases, Nash equilibria are determined under various assumptions
on the level of information.

\skp

We use the following notation. For $a,b\in\R$, the maximum [resp., minimum] is denoted by
$a\vee b$ [resp., $a\w b$], and $a^+=a\vee0$, $a^-=(-a)\vee0$.
For $x,y\in\R^k$ ($k$ a positive integer), $x\cdot y$
and $\|x\|$ denote the usual scalar product and $\ell_2$ norm, respectively.
Write $\{\be_i\}$, $i=1,\ldots,k$ for the standard basis in $\R^k$
and $1$ for $\sum_{i=1}^k\be_i$.
Denote $\R_+=[0,\iy)$. For $f:\R_+\to\R^k$, $\|f\|_T=\sup_{t\in[0,T]}\|f(t)\|$,
and, for $\theta>0$,
\[
w_T(f,\theta)=\sup_{0\le s<u\le s+\theta\le T}\|f_u-f_s\|.
\]
For a Polish space $\calS$, let $\C_\calS([0,T])$ and $\D_\calS([0,T])$
denote the set of continuous and, respectively, cadlag functions $[0,T]\to\calS$.
Write $\C_\calS$ and $\D_\calS$ for the case where $[0,T]$ is replaced by $\R_+$.
Endow $\D_\calS$ with the Skorohod $J_1$ topology. Write $X_n\To X$ for convergence in
distribution. A sequence of processes $X_n$ with sample paths in
$\D_\calS$ is said to be {\it $C$-tight} if it is tight and every subsequential
limit has, with probability 1, sample paths in $\C_\calS$.
For a sequence of processes $\xi^n$, $n\in\N$,
with sample paths in $\D_{\R^k}$, $C$-tightness is characterized (see VI.3.26 of \cite{jac-shi}) by
\begin{itemize}
\item[C1.] The sequence of random variables
$\|\xi^n\|_T$ is tight for every fixed $T<\iy$, and
\item[C2.] For every $T<\iy$,
$\eps>0$ and $\eta>0$ there exist $n_0$ and $\theta>0$ such that
\[
n\ge n_0 \text{ implies } \PP(w_T(\xi^n,\theta)>\eta)<\eps.
\]
\end{itemize}

For a positive integer $k$, $m\in\R^k$ and a symmetric, positive matrix $A\in\R^{k\times k}$,
{\it an $(m,A)$-Brownian motion} (BM) is a $k$-dimensional BM starting from zero,
having drift $m$ and infinitesimal covariance matrix $A$.

This paper is organized as follows. The model and the equilibrium result
appear in Section \ref{sec2}. Section \ref{sec3} and \ref{sec4} analyze the behavior
of the system under FP and SLQ, respectively,
and along the way also obtain diffusion limit results, that may be interesting by
their own right. Section \ref{sec5} addresses
RSP in these two settings, and proves the main result.

\section{\textbf{Model and Main Result}}\label{sec2}

We start by introducing the probabilistic model and the HW scaling.
Then we provide the game theoretic setting, and state the main result.

A sequence of queueing models is considered, indexed by $n \in \mathbb{N}$. The $n$th system has $N$ buffers and $n$ identical servers. Customers from $N$ distinct classes arrive at the system and, upon arrival, each customer is informed about the queue length at the buffer that corresponds to
its own customer class,
and, based on this information only, makes a decision whether to join or leave the system.
If a customer of class $i$ decides to join, he goes directly for service on the event that any of the servers is available, and otherwise he is queued in buffer $i$.
As far as the service policy is concerned, we consider FP and SLQ (that is, however, unknown to the customers). In the first case, the servers serve according to the rule given by $1>2>\cdots>N$. Thus, when a server becomes available, it admits into service a customer in the buffer with highest priority (that is,
least index) among all buffers that are non-empty at that instant.
Under SLQ, the buffer that currently has most customers receives highest priority
(where ties are broken arbitrarily).
At each buffer, the customers are always taken from the head of the line.
We assume the non-idling condition, that is, that
no server will idle as long as any customers are in the queue.

Let $(\Om,\mathcal{F}, \mathbb{P})$ be a probability space, on which all the random variables (r.v.s)
introduced below are to be defined. The arrivals in each class occur according to independent renewal processes. Let parameters $\lambda_i^n >0$, $i\in \{1,2,\ldots, N\}$, be given, representing the mean inter-arrival times of class-$i$ customers in the $n$th system. Let $\{\IA_i(l):l\in \mathbb{N}\}_i$ be independent sequences of strictly positive i.i.d.\ r.v.s with mean $1$ and variance $C^2_{\IA_i}$.
Let
\begin{align}
E^n_i(t)=\sup\Big\{l\geq 0:\sum_{k=1}^l\frac{\IA_i(k)}{\lambda^n_i}\leq t\Big\}, \qquad
t\geq 0\,.
\end{align}
Then $E^n_i$ counts the number of class-$i$ arrivals up to time $t$.
The parameters $\lambda_i^n$ satisfy
\begin{align}
\lambda_i^n=n\lambda_i+\sqrt{n}\hat{\lambda}_i+o(\sqrt{n})\,,
\end{align} where $\lambda_i>0$ and $\hat{\lambda_i}\in \mathbb{R}$ are fixed.
The service times of class-$i$ customers are assumed to be exponential with mean $\mu_i$.
The potential service processes, denoted by $\{S_i\}_{i=1,2,\ldots,N}$,
are thus assumed to comprise a collection of $N$ mutually independent
Poisson process, with rates $\mu_i$, $i=1,2,\ldots,N$, respectively.
They are assumed to have right-continuous sample paths.
While the arrival rates are accelerated with $n$, the individual service
rates are not. However, the capacity of the service pool grows due to the
increase of the number of servers, $n$. The resulting traffic intensity is thus
asymptotically given by $\sum_i\rho_i$, where $\rho_i=\la_i/\mu_i$.
We will assume the following critical load condition:
\begin{equation}\label{11}
\sum_i \rho_i=1.
\end{equation}

The initial conditions,
\[
Q^n(0)=(Q^n_1(0),Q^n_2(0),\ldots,Q^n_N(0)),\qquad
\Ps^n(0)=(\Ps^n_1(0),\Ps^n_2(0),\ldots,\Ps^n_N(0)),
\]
are $\Z_+^N$-valued r.v.s representing
the number of customers initially in the buffers and in service, respectively.
It is assumed that the initial configuration satisfies
$1\cdot Q^n(0)>0$ implies $1\cdot\Ps^n(0)=n$, reflecting the non-idling condition.

For each $n$, the three objects
\begin{equation}
  \label{10}
  \{E^n_i\}_i,\qquad \{S_i\}_i,\qquad (Q^n(0),\Ps^n(0))
\end{equation}
are assumed to be mutually independent.
The triplet \eqref{10} will be referred to as the {\it stochastic primitives}
of the model. All r.v.s introduced below, describing the system dynamics,
will be given as functions of the stochastic primitives
and of the collection of decisions taken by the strategic customers.

Thus, before describing the system dynamics, we introduce the notation
for the decision variables. The customers initially in the system do not participate in
the game formulation, and therefore in what follows, unless otherwise stated,
the term {\it customer} will refer to those customers that arrive after time zero.
A customer will be identified by a pair $(i,j)$, where $i\in\{1,2,\ldots,N\}$
is its class, and $j\in\N$ is its serial number in order of arrival.
The collection of decision variables $\del=\{\delta_{ij}: i\in\{1,2,\ldots,N\}, j\in \mathbb{N}\}$,
where $\del_{ij}\in\{0,1\}$, specifies the decision of each of the customers.
Having $\delta_{ij}=1$ [resp., $0$] specifies that the $j$th class-$i$ customer to arrive
decides to join [resp., leave] the system. Let
\begin{equation}\label{40}
J^n_i(t)=\sum_{j=1}^{E^n_i(t)}\delta_{ij},\qquad
R^n_i(t)=\sum_{j=1}^{E^n_i(t)}(1-\delta_{ij})
\end{equation}
denote counting processes for joining and reneging customers.
Let $Q^n_i(t)$ be the number of class-$i$ customers waiting at the $i$th buffer at time $t$,
and let $B^n_i(t)$ be the number of class-$i$ customers routed to the service pool
by that time. Then we have
\begin{align}\label{queuelength}
Q^n_i(t)=Q^n_i(0)+E^n_i(t)-B^n_i(t)-R^n_i(t)\,.
\end{align}
Let $\Ps^n_i(t)$ denote the number of class-$i$ customers in service at time $t$. Then
\begin{align}\label{19}
\Ps_i^n(t)=\Ps_i^n(0)+B^n_i(t)-D_i^n(t)\,,
\end{align}
where the departure process $D^n_i$ counts the number of completed services of class-$i$
jobs since time $0$ (including initial customers).
It is assumed that the departure process is given, in terms of the potential
service process, by
\begin{equation}\label{20}
D^n_i(t)=S_i\Big(\int_0^t\Ps^n_i(u)du\Big)\,.
\end{equation}
The non-idling condition is expressed by requiring
\begin{equation}
  \label{23}
  \text{for every $t$, } 1\cdot Q^n(t)>0 \text{ implies } 1\cdot\Ps^n(t)=n.
\end{equation}
Under the FP policy we have
\begin{equation}\label{24}
\int_{[0,\iy)}\sum_{k=1}^{i-1}Q^n_k(t)dB^n_i(t)=0,\qquad i=2,3,\ldots,N.
\end{equation}
And under SLQ, a server that becomes available at time $t$
chooses class $i_0$, where
$i_0\in\argmax_iQ^n_i$ (where ties are broken in an arbitrary, but concrete way), namely
\begin{equation}\label{24+}
\int_{[0,\iy)}1_{\{Q^n_i(t-)<\max_k Q^n_k(t-)\}}dB^n_i(t)=0,\qquad i=1,2,\ldots,N.
\end{equation}
The collection of equations \eqref{40}--\eqref{23} and either \eqref{24} or \eqref{24+},
along with the primitives and the decision variables $\del$, uniquely define the processes
$Q^n$, $X^n$, $\Ps^n$, $B^n$ and $D^n$ under each of the two policies.
Note that these processes are right-continuous by construction.

Now let
\begin{equation}\label{16}
\JT^n_i(t)=\inf\{s\ge t:J_i^n(s)>J_i^n(t-)\},
\end{equation}
(where, by convention, $\JT^n_i(0-)=0$),
represent the time of arrival of the first class-$i$ customer to join the system at or after time $t$.
Let also
\begin{equation}\label{17}
\RT^n_i(t)=\inf\{s>t:B^n_i(s)\geq B^n_i(\JT^n_i(t))+Q^n_i(\JT^n_i(t))\}\,.
\end{equation}
Then $\RT^n_i(t)$ gives the time when the customer joining at $\JT^n_i(t)$
enters service.
The time that particular customer waits in the queue is then given by
\begin{equation}\label{18}
\WT^n_i(t)=\RT^n_i(t)-\JT^n_i(t)\,.
\end{equation}
Note that, as a consequence,
\begin{align}\label{31}
Q^n_i(\JT^n_i(t))=B^n_i(\JT^n_i(t)+\WT^n_i(t))-B^n_i(\JT^n_i(t))\,.
\end{align}
(JT, RT, WT as well as AT defined below, are mnemonics for joining time, routing time, waiting
time and arrival time.)
We shall also need notation of arrival time and waiting time of the $j$th class-$i$
customer. These are obtained as follows:
\[
\AT^n_{ij}=\text{inv}E^n_i(j)=\inf\{t\ge0:E^n_i(t)\ge j\},
\]
\[
\WT^n_{ij}=\WT^n_i(\AT^n_{ij}).
\]
Note that while $\WT^n_{ij}$ is well-defined for all $(i,j)$, it only gives
the waiting time for those customers $(i,j)$ that have actually joined
the system; this concept is indeed meaningless for the reneging customers.
Scaled versions of the main stochastic processes introduced above
are defined as follows:
\begin{align}
\label{48}
\hat{Q}^{n}_i(t) &= \frac{Q^{n}_i(t)}{\sqrt{n}}\,\,,\qquad\hat{B}^{n}_i(t)=  \frac{B^{n}_i(t)-n\lambda_i t}{\sqrt{n}}\,,\\
\notag
\hat{R}^{n}_i(t) &= \frac{R^{n}_i(t)}{\sqrt{n}}\,\,,\qquad\hat{S}^{n}_i(t)= \frac{S_i(nt)-n\mu_i t}{\sqrt{n}}\,,
\qquad\hat{D}^n_i(t)=\hat{S}^{n}_i\Big(\int_0^t\bar{\Ps}^n_i(u)du\Big)\,,
\\
\notag
\hat{E}^{n}_i(t) &= \frac{E^{n}_i(t)-\lambda^n_i t}{\sqrt{n}}\,\,,\qquad\hat{\Ps}^{n}_i(t)= \frac{\Ps^{n}_i(t)-\rho_i n}{\sqrt{n}}\,.
\end{align}
Also define,
\begin{equation}\label{15}
\wh\WT^n_i(t)=\sqrt n\WT^n_i(t),\qquad \wh{\WT}^n_{ij}=\sqrt n \WT^n_{ij}.
\end{equation}
It is assumed that the scaled initial condition converges in distribution:
\begin{equation}\label{22}
(\hat Q^n(0),\hat\Ps^n(0))\To (0,\Ps(0)),
\end{equation}
where $\Ps(0)$ is an $\R^N$-valued r.v.\ with $\sum_i\Ps_i(0)\le0$.

This completes the description of the stochastic processes of interest.
We denote the collection of processes, that we will sometimes refer to as {\it dynamics}, by
\[
\calS^n=\calS^n[\del]=(J^n,R^n,Q^n,B^n,\Ps^n,D^n,\JT^n,\RT^n,\WT^n),
\]
where we emphasize the dependence of these processes on the decision
variables $\del$. We will use similar notation to emphasize the dependence of
each of the components of $\calS^n$ on $\del$, as for example $Q^n[\del]$.

Now we come to the game-theoretic setting. It is described for fixed $n$.
In the game, the
dynamics described above will serve as the game's state. The game is played by the customers
to arrive up to time $\bar{T}$, where $\bar{T}\in(0,\iy)$ is fixed throughout.
A decision is made by each customer once the queue length of the corresponding class
at the time of arrival is revealed to it.
Thus for our purpose, a {\it strategy} is a mapping $\sig:\Z_+\to\{0,1\}$.
We denote the set of all such mappings by $\Sig$.
A {\it strategy profile} is an element of $\bar\Sig:=\Sig^{\{1,2,\ldots,N\}\times\N}$.
Let a strategy profile $\sig=\{\sig_{ij}\}\in\bar\Sig$ be given.
We say that {\it the game is played with the strategy profile $\sig$}
if one has
\begin{equation}\label{13}
\begin{cases}
  &\calS^n=\calS^n[\Del^n], \text{ (specifically, $Q^n=Q^n[\Del^n]$)}, \\
  &\Del^n_{i,j}=\sig_{ij}(Q^n_i(\AT^n_{ij}-)),\qquad i\in\{1,2,\ldots,N\},\, j\in\N.
\end{cases}
\end{equation}
Thus $\calS^n$ is the dynamics resulting from having each customer $(i,j)$
adopt the strategy $\sig_{ij}$, and
$\Del^n_i(j)$ is a r.v.\ representing the action taken by customer $(i,j)$
in that situation.
An argument by induction on the times of arrival shows that the system of equations
\eqref{13} has a unique solution, and thus $\calS^n$ and $\Del^n$ are well-defined
r.v.s. We will also need a notation for the dynamics $\calS^n$, thus determined
by \eqref{13}, as a function of the strategy profile $\sig$. We write it as $\calS^n(\sig)$.

We formulate the payoff for customer $(i,j)$ by accounting for a cost associated
with not receiving service (in case of reneging) and a function of the waiting
time (in case of joining). To this end, we are given constants $r_i>0$, $i\in\{1,2,\ldots,N\}$
and functions $h_i:\R_+\to\R_+$, assumed to be continuous, strictly increasing
and to vanish at zero.
For a strategy profile $\sig=\{\sig_{ij}\}$, denote
$\sig^{ij}=\{\sig_{k,l}:(k,l)\ne(i,j)\}$.
The payoff for customer $(i,j)$, when the strategy profile $\sig$ is played, is given by
\begin{align}\label{14}
C^{n}_{ij}(\sig_{ij},\sig^{ij})=\begin{cases}
r_i, & \Del_i^{n}(j)=0,\, \AT^n_{ij}\le \bar{T},\\
h_i(\wh{\WT}^{n}_{ij}), & \Del_i^{n}(j)=1,\,\AT^n_{ij}\le \bar{T},\\
0, & \AT^n_{ij}>\bar T.
\end{cases}
\end{align}
Thus, according to the payoff definition, the game
neglects all customers arriving after time $\bar T$.

For fixed $n$ and $\eps>0$, and an event $\hat\Om\in\mathcal{F}$,
a strategy profile $\sig=\{\sig_{ij}\}$ is said to be an
{\it $\eps$-Nash equilibrium on the event $\hat\Om$} if
\begin{equation}\label{21}
\forall (i,j),\quad \forall\tau\in\Sig,\quad
C^n_{ij}(\sig_{ij},\sig^{ij})\le C^n_{ij}(\tau,\sig^{ij})+\eps
\end{equation}
holds on $\hat\Om$.
A sequence of strategy profiles $\{\sig^n\}_{n\in\N}$ is said to be
an {\it $\eps$-Nash equilibrium w.h.p.}, if
there exist events $\hat{\Om}^n$, $n\in\N$, such that, for every $n$,
$\sig^n$ is an $\eps$-Nash equilibrium on $\hat{\Om}^n$, and $\PP(\hat{\Om}^n)\to1$
as $n\to\iy$.

For each $n$ and $(i,j)$, consider the strategy
\begin{equation}\label{12}
\sig^n_{ij}(q)=\begin{cases}
  1, & \text{if } h_i\Big(\frac{q}{\sqrt n \la_i}\Big) \le r_i,\\
  0, & \text{otherwise,}
\end{cases}
\qquad q\in\Z_+.
\end{equation}

\begin{thm}\label{th1}
For any $\eps>0$, under each of the two scheduling policies defined above,
the sequence of strategy profiles $\{\sig^n\}$ defined in \eqref{12} is an
$\eps$-Nash equilibrium w.h.p.
\end{thm}

Toward proving this result, we analyze the diffusion scale processes, and, along the way
also obtain diffusion limit results. These are Proposition \ref{prop1}, for FP,
and Proposition \ref{prop2}, for SLQ.

\begin{rem}\label{rem1}
{\it RSP does not always hold.}
\rm
One of the main issues we address is the validity of RSP under
the scheduling policies considered. In order to prove the main result, this principle
needs to hold in a strong form, namely that,
w.h.p., {\it every} customer arriving, and joining, in the given time interval $[0,\bar T]$,
experiences a delay given, with high precision,
by the ratio between queue length and arrival rate.
It should be noted that this property is not valid for arbitrary scheduling. For example,
consider a scheduling that prioritizes class 1 over class 2 up to a certain
fixed time, $t_0$, and then switches to the a priority of 2 over 1.
The standard prediction is that the diffusion scale waiting time for a class-$i$
customer is approximately given by
$\wh\WT\approx \la_i^{-1}\hat Q_i=(\rho_i\mu_i)^{-1}\hat Q_i$, where $\hat Q_i$ is the
diffusion scale queue length at the arrival time.
Now, consider a class-2 customer present in the buffer at time $t_0$.
Such a customer will be sent to service
approximately $(\rho_1\mu_1+\rho_2\mu_2)^{-1}\hat q$ units of time after $t_0$, where
$\hat q=n^{-1/2}q$, and $q$ is its position in line
at $t_0$, because when 2 has priority,
every server in the pool to become available will pick a customer from buffer 2.
Hence, w.h.p., most customers that are in buffer 2 at time $t_0$, that are, in fact,
$O(\sqrt n)$ in number, will experience a delay significantly
different than that predicted by RSP.
This number increases even further under a policy that
switches priority many times during the time interval
in question. While these policies may not be particularly interesting by their own right,
this discussion shows that there is content
in the assertion that the principle does hold for the policies of interest.
\end{rem}

\begin{rem}
{\it Individual decisions may have long term effect.}
\rm
The analysis must take into account the possible behavior of customers that do not follow the proposed
rule. At the technical level, the estimates that lead to existence of diffusion limits are
dealt with for different behaviors of customers.
It may seem that it is enough to consider the behavior of the system
when all customers follow the proposed rule, and then argue that the behavior of a single
customer will have a negligible effect.
It should be noted, however, that
the decision of one customer may affect significantly the waiting time of other customers.
As a simple example for that, consider a two-class system under FP,
where, at a certain time, a high priority customer arrives to find an empty buffer
of its own class. If he decides to leave, and for a little while there are no new arrivals,
then the first-in-line customer at the low priority class will get served as soon as
a server becomes available. If he joins, it is possible that a large number of high priority
customers will join soon after, so that the waiting time of the low priority customer referred
to above will delay considerably. Hence a single player's decision may have a significant effect
on other players.
\end{rem}

\section{\textbf{Fixed Priority}}\label{sec3}

This section is devoted to a convergence result in the case where the servers implement
the FP scheduling. It provides the main estimates that determine the limiting behavior of the fluid and diffusion scaled processes, that are later used to prove RSP.

Throughout, $\sig^n=\{\sig^n_{ij}\}$ denotes the strategy profile \eqref{12}.
Given $(i,j)$, denote by $\bar{\sigma}^n_{ij}\in\Sig$ the strategy
$\bar{\sigma}^n_{ij}=1-\sigma^n_{ij}$, that acts precisely as the negation of $\sigma^n_{ij}$.
We begin by noting that in order to show that $\sig^n$ is an $\eps$-Nash equilibrium w.h.p.,
it suffices to consider \eqref{21} with $\tau=\bar\sig^n_{ij}$ only.
Indeed, given $(i,j)$ and $\tau\in\Sig$,
define $A=\{q\in\Z_+:\tau(q)\neq \sigma_{ij}^n(q)\}$. Then we have
\begin{align*}
C^n_{ij}(\tau,\sig^{n,ij})=
\begin{cases}
C^n_{ij}(\bar{\sigma}^n_{ij},\sig^{n,ij}), &\mbox{if} \,\,Q^n_i(\AT^n_{ij}-)\in A,
\\
C^n_{ij}(\sigma^n_{ij},\sig^{n,ij}), & \mbox{if} \,\,Q^n_i(\AT^n_{ij}-)\in A^c,
\end{cases}
\end{align*}
and so the validity of \eqref{21} for $\tau=\bar\sig^n_{ij}$ and $\tau=\sig^n_{ij}$
(the latter being trivial) implies the validity of this inequality for $\tau\in\Sig$.

We will use the term {\it scenario} for the collection of processes obtained under any one of
the strategy profiles $(\bar\sig^n_{ij},\sig^{n,ij})$. More precisely,
let us fix $n$. Recall that, for $\sig\in\bar\Sig$,
$\calS^n(\sig)$ denotes the dynamics obtained when a strategy profile $\sig$ is played.
Let
$$
\SC=\{(i,j):i\in \{1,2,\ldots,N\},\,\,j \in \mathbb{N}\}\,.
$$
For $s=(i,j)\in \SC$, the {\it scenario} $s$ is defined to be
$\calS^n(\bar\sig^n_{ij},\sig^{n,ij})$,
namely the dynamics corresponding to player $(i,j)$ playing $\bar\sig^n_{ij}$ and all
other players $(k,l)$ playing $\sig^n_{kl}$. In addition, {\it scenario $0$},
that we will also call the {\it reference} scenario, is defined as $\calS^n(\sig^n)$.
Scenarios are thus indexed by the set $\SC_0:=\SC\cup\{0\}$.
As we have just argued, the main result will follow once we show that there exist
events $\hat{\Om}^n$ such that, for every $n$, on $\hat{\Om}^n$,
\begin{align}\label{30}
\forall (i,j)\quad
C^n_{ij}(\sig^n_{ij},\sig^{n,ij})\le C^n_{ij}(\bar{\sigma}^n_{ij},\sig^{n,ij})+\eps,
\end{align}
and $\PP(\hat{\Om}^n)\to1$ as $n\to\iy$.
We thus work in what follows with scenarios. In order to address all scenarios
simultaneously, the dependence of the processes on the scenario has to be reflected in the notation.
For each of the processes introduced above, except for the stochastic primitives and their scaled
versions, an additional superscript $s$ will indicate
that the process is considered under scenario $s\in\SC_0$.
Thus, for example, $Q^{n,s}=Q^n(\bar\sig^n_{ij},\sig^{n,ij})$ if $s=(i,j)$, and
$Q^{n,s}=Q^n(\sig^n)$ if $s=0$.

Throughout what follows, we adopt the convention that $e^{n,s}(t)$
(or sometimes $e^{n,s}_i(t)$), $t\in[0,T]$, denotes a generic family of processes, indexed by
$n\in\N$ and $s\in\SC_0$,
that can change from one appearance to another, and has the property
\begin{equation}\label{44}
\sup_s\| e^{n,s}\|_T\to0 \quad\text{in probability, as $n\to\iy$.}
\end{equation}

The balance equations \eqref{queuelength}, \eqref{19} and \eqref{20}
have the following form when translated to the diffusion scale, namely
\begin{align}\label{41}
&\hat{Q}^{n,s}_i(t)= \hat{Q}^n_i(0)+\hat{E}^{n}_i(t)-\hat{B}^{n,s}_i(t)-\hat{R}^{n,s}_i(t)
+n^{-1/2}(\lambda^n_i-n\lambda_i)t\,,\\ \label{42}
&\hat{\Ps}^{n,s}_i(t)=\hat{\Ps}^n_i(0)+\hat{B}^{n,s}_i(t)
-\hat{S}^{n}_i\Big(\int_0^t\bar{\Ps}^{n,s}_i(u)du\Big)-\mu_i\int_0^t\hat{\Ps}^{n,s}_i(u)du\,.
\end{align}
Let $X^{n,s}_i=Q^{n,s}_i+\Ps^{n,s}_i$ represent the total number of class-$i$ customers
in the system, and let its scaled version be defined by
\begin{align}\label{43}
\hat{X}^{n,s}_i(t)=\frac{X^{n,s}_i(t)-\rho_i n}{\sqrt n}
=\hat{Q}^{n,s}_i(t)+\hat{\Ps}^{n,s}_i(t)\,.
\end{align}
Then by the assumptions on the initial conditions we have
$$
\hat{X}^n(0)\rightarrow \Ps(0)=: X_0\,.
$$
Our first estimate addresses the scaled queue lengths of the high priority classes.

\begin{lem}\label{implem}For $i=1,2,\ldots, N-1$ and for any $T < \infty$ we have
$$
\displaystyle\sup_s\|\hat{Q}^{n,s}_i\|_T\rightarrow 0,
\quad\text{in probability.}
$$
\end{lem}

\begin{proof}
By the functional central limit theorem,
\begin{align}\label{fclt}
(\hat{E}^n, \hat{S}^n)\Rightarrow (W_1,W_2)\,,
\end{align}
where $W_1$ and $W_2$ are independent $N$-dimensional BMs, with $W_1$ a $(0,A_1)$-BM
and $W_2$ a $(0,A_2)$-BM,
$A_1=\diag(\lambda_iC^2_{\IA_i})$,
and $A_2=\diag(\mu_i)$ (see Section 17 of \cite{Bill}). In particular,
the sequence $(\hat E^n,\hat S^n)$ is $C$-tight.

Fix $\epsilon > 0$. Define the event
$$
\Om^n=\Big\{\sum_{i=1}^{N-1}\hat{Q}^{n}_i(0)\leq \frac{\epsilon}{4}\quad\mbox{and}\quad\bar{\Ps}^{n}_i(0) \geq \rho_i - \frac{\eps_i}{4}\quad\mbox{for all}\quad i\in \{1,2,\ldots ,N-1\}\Big\},
$$
where
$\eps_i=\frac{\epsilon}{\mu_i (N-1)}$.
Then by the assumption \eqref{22} on the initial conditions we have $\mathbb{P}(\Om^n)\rightarrow 1$. For $s \in\SC_0$ define
\begin{align*}
\tau_1^{n,s}=\inf\Big\{t\geq 0:\sum_{i=1}^{N-1}\hat{Q}^{n,s}_i(t)\geq \epsilon\,\,\mbox{ or }\,\,\bar{\Ps}^{n,s}_i(t) \leq \rho_i - \eps_i\mbox{ for some }\, i\in \{1,2,\ldots ,N-1\}\Big\}\,.
\end{align*}
Let $A^{n,s}=\{\tau_1^{n,s}\leq T\}$.
Now let
$$
A_1^{n,s}=\Big\{\omega \in A^{n,s}:\sum_{i=1}^{N-1}\hat{Q}^{n,s}_i(\tau_1^{n,s})\geq \epsilon\Big\}\cap\Om^n,
$$
$$
A_2^{n,s,i}=\Big\{\omega \in A^{n,s}: \sum_{k=1}^{N-1}\hat{Q}^{n,s}_k(\tau_1^{n,s}) < \epsilon\,\, \mbox{ and }\,\, \bar{\Ps}^{n,s}_i(\tau_1^{n,s}) \leq \rho_i - \eps_i \Big\}\cap\Om^n,
\quad i\le N-1.
$$
For $\omega \in A_1^{n,s}$ there exists $\sigma_1^{n,s}=\sigma_1^{n,s}(\om)$ such that
\begin{align}\label{25}
\sum_{i=1}^{N-1}\hat{Q}^{n,s}_i(\sigma_1^{n,s})\leq \frac{\epsilon}{2},\quad\mbox{and, on
 }\,\,I^{n,s}_1:=[\sigma_1^{n,s},\tau_1^{n,s}],\quad\sum_{i=1}^{N-1}\hat{Q}^{n,s}_i >0\,.
\end{align}
Throughout, for $0\le t_1\le t_2<\iy$, $I=[t_1,t_2]$ and $f:\R_+\to\R$, we use the notation
\[
f[t_1,t_2]=f[I]=f(t_2)-f(t_1).
\]
By \eqref{queuelength} and the fact that $R^n_i$ is nondecreasing, we have
on $A_1^{n,s}$
\begin{equation}\label{26}
\frac{\epsilon\sqrt n}{2}\leq
\sum_{i=1}^{N-1}Q^{n,s}_i[I_1^{n,s}]
\leq \sum_{i=1}^{N-1}E^n_i[I_1^{n,s}]
-\sum_{i=1}^{N-1}B^{n,s}_i[I_1^{n,s}].
\end{equation}
By \eqref{25} and \eqref{23}, $1\cdot\Ps^{n,s}(t)=n$ for $t=\sigma_1^{n,s}$
and $t=\tau_1^{n,s}$. Thus by \eqref{19}, $1\cdot B^{n,s}[I_1^{n,s}]
=1\cdot D^{n,s}[I_1^{n,s}]$.
Moreover, since by \eqref{25} the high priority buffers are non-empty on the time interval
of interest,
the priority rule expressed by \eqref{24} dictates that the process $B^{n,s}_N$ does not increase
on that interval. As a result, the last term in \eqref{26} equals
$1\cdot D^{n,s}[I_1^{n,s}]$, and
\[
\frac{\epsilon}{2}\leq \sum_{i=1}^{N-1}\hat{E}^n_i[I_1^{n,s}]
+\sum_{i=1}^{N-1}\frac{\lambda^n_i(\tau_1^{n,s}-\sigma_1^{n,s})}{\sqrt{n}}
-\sum_{i=1}^{N}\hat{D}^{n,s}[I_1^{n,s}]
-n^{-1/2}\sum_{i=1}^{N}\mu_i\int_{\sigma_1^{n,s}}^{\tau_1^{n,s}}\Ps^{n,s}_i(u)du\,.
\]
On the time interval under consideration we have for $i < N$ that
$\Ps^{n,s}_i \geq n\del_i$, where $\del_i=\rho_i-\eps_i$.
Thus, denoting $\mumin=\min_i\mu_i>0$,
\begin{align*}
\sum_{i=1}^N\mu_i\Ps^{n,s}_i&=\sum_{i=1}^{N-1}\mu_i(n\delta_i+\Ps^{n,s}_i-
n\delta_i)+\mu_N\Ps^{n,s}_N\\
&\geq n\biggl(\sum_{i=1}^{N-1}\lambda_i-\epsilon\biggr)+\mumin\sum_{i=1}^{N-1}(\Ps^{n,s}_i-
n\delta_i)+\mumin\Ps^{n,s}_N\\
&=n\biggl(\sum_{i=1}^{N-1}\lambda_i-\epsilon
+\mumin\rho_N+\mumin\eps_i\biggr),
\end{align*}
where the last equality uses the fact that $\sum_{i=1}^N\Ps^{n,s}_i=n$ that is true thanks to
the non-idling condition \eqref{23} and the fact that, by \eqref{25}, the queues
are not all empty.
Therefore for $\epsilon$ small enough there exists a $\delta > 0$, such that
$$
\sum_{i=1}^{N}\mu_i\int_{\sigma_1^{n,s}}^{\tau_1^{n,s}}\Ps^{n,s}_i(u)du
\geq n(\sum_{i=1}^{N-1}\lambda_i+\delta)(\tau_1^{n,s}-\sigma_1^{n,s})\,.
$$
Hence we have
\begin{align*}
\frac{\epsilon}{2}&\leq \sum_{i=1}^{N-1}\hat{E}^n_i[I_1^{n,s}]
-\sum_{i=1}^{N}\hat{D}^{n,s}[I_1^{n,s}]\\
&\qquad+\sum_{i=1}^{N-1}
\frac{(\lambda^n_i-n\lambda_i)(\tau_1^{n,s}-\sigma_1^{n,s})}{\sqrt{n}}
-\sqrt{n}\delta(\tau_1^{n,s}-\sigma_1^{n,s})\,.
\end{align*}
Let $r_n>0$ be a sequence such that $r_n\rightarrow 0$ and
$\sqrt{n}r_n\rightarrow \infty$. If $\tau_1^{n,s}-\sigma_1^{n,s}\leq r_n$ then
\begin{align*}
\frac{\epsilon}{2}&\leq \sum_{i=1}^{N-1}w_T(\hat{E}^n_i,r_n)+\sum_{i=1}^{N}w_T(\hat{S}^n_i,r_n)+Kr_n\,,
\end{align*}where $K$ is a constant and, throughout, for $f:\R_+\to\R^k$ ($k$ a positive integer),
\[
w_T(f,a)=\sup\{\|f(t)-f(s)\|:s,t\in[0,T],\,|t-s|\le a\},\qquad a>0.
\]
On the other hand, if $\tau_1^{n,s}-\sigma_1^{n,s}> r_n$ then
\begin{align*}
\frac{\epsilon}{2}&\leq 2\sum_{i=1}^{N-1}\big\|\hat{E}^n_i\big\|_T + KT + 2\sum_{i=1}^{N-1}\big\|\hat{S}^n_i\big\|_T - \sqrt{n}\delta r_n\,.
\end{align*}
Hence by \eqref{fclt} and the resulting $C$-tightness of $\hat{E}^n_i$ and $\hat{S}^n_i$,
we have
\begin{equation}\label{32}
\mathbb{P}\bigl(\cup_s A^{n,s}_1\bigr)\rightarrow 0,\,\,\mbox{as}\,\,n\rightarrow \infty\,.
\end{equation}

Next, on $A^{n,s,i}_2$, for $i\le N-1$ fixed, again there exists a time
$\sigma^{n,s}_2=\sigma^{n,s}_2(\om)$ such that
$$
\bar{\Ps}^{n,s}_i(\sigma^{n,s}_2)\geq \rho_i-\frac{\eps_i}{2}\,\quad
\mbox{and, on}\,\,I^{n,s}_2:=[\sigma^{n,s}_2,\tau^{n,s}_1],\quad\bar{\Ps}^{n,s}_i\leq \rho _i\,.
$$
Thus $X^{n,s}_i[I^{n,s}_2]\leq \sqrt{n}\epsilon - \frac{n\eps_i}{2}$, or
\[
E^{n}_i[I^{n,s}_2]-D^{n,s}_i[I^{n,s}_2]
-R^{n,s}_i[I^{n,s}_2]\leq \sqrt{n}\epsilon - \frac{n\eps_i}{2},
\]
and therefore
\[
\hat{E}^{n}_i[I^{n,s}_2]-\hat{D}^{n,s}_i[I^{n,s}_2]
+\frac{(\lambda^n_i-n\lambda_i)(\tau_1^{n,s}-\sigma_1^{n,s})}{\sqrt{n}}\leq
\epsilon - \frac{\sqrt{n}\eps_i}{2}+\frac{1}{\sqrt{n}},
\]
whence
\[
-2\big\|\hat{E}^n_i\big\|_T-\big\|\hat{S}^n_i\big\|_T-KT \leq \epsilon
 - \frac{\sqrt{n}\eps_i}{2}+\frac{1}{\sqrt{n}}.
\]
Therefore by the tightness
of $\big\|\hat{E}^n_i\big\|_T$ and $\big\|\hat{S}^n_i\big\|_T$, $n\in\N$ (for $T$ fixed),
we have
\begin{equation}\label{33}
\mathbb{P}\bigl(\cup_s A^{n,s,i}_2\bigr)\rightarrow 0,\,\,\mbox{as}\,\,n\rightarrow \infty\,,
\quad i\le N-1.
\end{equation}
Putting together the estimates \eqref{32} and \eqref{33},
we obtain
$\mathbb{P}(\cap_s (A^{n,s})^c)\rightarrow 1$.
Since $\eps>0$ is arbitrary, the result follows.
\end{proof}

Define
\begin{equation}\label{35}
\theta_i = \lambda_i h_i^{-1}(r_i)\,,
\end{equation}
and note that these constants are positive. By \eqref{12}, under the reference scenario,
class-$i$ customers always renege when
the scaled queue length $\hat{Q}^n_i$ is in the interval $(\theta_i,\theta_i+\frac{1}{\sqrt{n}}]$ and therefore
the scaled queue length never exceeds that bound. Under any other scenario,
there is at most one customer that does not follow the rule \eqref{12}, and so we have
\begin{equation}
  \label{36}
  \hat Q^{n,s}_i(t)\le\theta_i+2n^{-1/2},\qquad t\ge0,\, n\in\N,\, s\in\SC_0,\, i=1,\ldots,N.
\end{equation}
Conversely, a class-$i$ reneging will never take place when
$h_i(\hat Q^{n,s}_i/\la_i)<r_i$, except, possibly, by a single customer.

\begin{lem}\label{implem1}
i.
For $i=1,2,\ldots,N-1$,
\begin{equation}\label{37}
\sup_s\hat R^{n,s}_i(T)\to0,\quad \text{in probability, as $n\to\iy$}.
\end{equation}
ii. For $i=1,\ldots,N$,
\begin{equation}\label{38}
\sup_s\|\bar{\Ps}^{n,s}_i-\rho_i\|_T\rightarrow 0,\quad \text{in probability, as $n\to\iy$}.
\end{equation}
\end{lem}

\begin{proof}
i. By the discussion preceding the Lemma, $R^{n,s}_i(T)\le1$ on the event that
$\|\hat Q^{n,s}_i\|_T<\theta_i$. Hence \eqref{37} follows from Lemma \ref{implem}.

ii. We begin by proving the result for the high-priority classes.
Thus, fix $i\le N-1$.
We have by \eqref{queuelength}, \begin{align*}\bar{Q}^{n,s}_i(t)&=\bar{Q}^{n}_i(0)+\bar{E}^n_i(t)-\bar{B}^{n,s}_i(t)-\bar{R}^{n,s}_i(t)\\
&=\bar{Q}^{n}_i(0)+(\bar{E}^n_i(t)-\lambda_i t)-(\bar{B}^{n,s}_i(t)-\lambda_i t)-\bar{R}^{n,s}_i(t)\,.
\end{align*}
By the functional law of large numbers, $\sup_{0\le t\le T}|\bar E^n_i(t)-\la_it|\to0$ in probability.
Hence the estimates of Lemma \ref{implem} give (recall the convention \eqref{44})
\begin{equation}\label{34}
\bar{B}^{n,s}_i(t)=\lambda_i t+e^{n,s}_t.
\end{equation}
Next, by \eqref{19} and \eqref{20},
using the identity $\rho_i=\la_i/\mu_i$,
{\small\begin{align*}
\bar{\Ps}^{n,s}_i(t)-\rho_i&=\bar{\Ps}^n_i(0)-\rho_i+(\bar{B}^{n,s}_i(t)-\lambda_it)
-\frac{1}{n}\Big[S_i\Big(\int_0^tn\bar{\Ps}^{n,s}_i(u)du\Big)-\mu_i\int_0^tn\bar{\Ps}^{n,s}_i(u)du\Big]
\\&\quad-\mu_i\int_0^t(\bar{\Ps}^{n,s}_i(u)-\rho_i)du.
\end{align*}
Using the fact $\bar\Ps^{n,s}_i\le 1$ we have, for $t\in[0,T]$,
\begin{align*}
|\bar{\Ps}^{n,s}_i(t)-\rho_i| &\leq |\bar{\Ps}^n_i(0)-\rho_i|
+\beta^n_T+n^{-1/2}\|\hat{S}^n_i\|_T
+\mu_i\int_0^T\sup_s\sup_{0\leq r\leq u}|\bar{\Ps}^{n,s}_i(r)-\rho_i|du.
\end{align*}}
It follows from \eqref{22} that $\|\bar\Ps^n(0)-\rho\|\to0$ and from the law of large numbers
for the Poisson process, that $n^{-1/2}\|\hat S^n_i\|_T\to0$, in probability.
Using these facts along with \eqref{34}, the result \eqref{38}, for $i\le N-1$,
follows upon applying Gronwall's lemma.

Next we consider the class $N$.
Because $\sum \rho_i =1$ and $\sum\bar{\Ps}^{n,s}_i \leq 1$, we have from the validity
of \eqref{38} for $i\le N-1$,
$$
\sup_s\sup_{0\leq t\leq T}(\bar{\Ps}^{n,s}_N(t)-\rho_N)^+\rightarrow 0
$$
in probability as $n\rightarrow \infty$.
Using this and the assumption on the initial conditions, the probability of
$\Om^n_1:=\{\gamma^n<\eps/16\}\cap\{|\bar\Ps^n_N(0)-\rho_N|<\eps/2\}$ converges to 1,
where $\gamma^n=\displaystyle\sup_s\sum_{i\le N-1}\|\bar\Ps^{n,s}_i-\rho\|_T$.
Now let
$$
\Om^{n,s}=\{\omega:\inf_{0\leq t\leq T}\bar{\Ps}^{n,s}_N(t) \le \rho_N-\epsilon\}\,.
$$
Then for $\omega \in\Om^{n,s}\cap\{|\bar\Ps^n_N(0)-\rho_N|<\eps/2\}$,
there exist times $0\le\sigma^{n,s}_3(\om)\le\tau^{n,s}_3(\om)\le T$
such that
\begin{align*}
\bar{\Ps}^{n,s}_N(\sigma^{n,s}_3)> \rho_N-\frac{\epsilon}{2}\,\,,\,\,
\bar{\Ps}^{n,s}_N(\tau^{n,s}_3)\leq \rho_N-\eps\,\,\mbox{ and }\,\,
\bar{\Ps}^{n,s}_N(t)\leq \rho_N-\frac{\epsilon}{8}\,\,\text{ for all }\,\,
t\in I^{n,s}:=[\sigma_3^{n,s},\tau^{n,s}_3]\,.
\end{align*}
Also, on the event $\Om^{n,s}\cap\{\gamma^n<\eps/16\}$,
\begin{align*}
\sum_{i=1}^{N-1}\bar{\Ps}^{n,s}_i(t)\leq \sum_{i=1}^{N-1}\rho_i + \frac{\epsilon}{16} \,\,
\text{ for all }\,t\in I^{n,s}.
\end{align*}
Thus on $I^{n,s}$ we have $\sum_{i=1}^{N}\bar{\Ps}^{n,s}_i(t)\leq 1-\epsilon/16 < 1$,
which implies by the non-idling assumption that, on this time interval,
we have $\sum_{n=1}^N Q^{n,s}_i(t)=0$.
As a result, on this time interval there is no reneging under the reference scenario,
and there is at most one reneging under any other scenario.
Recalling that $X^{n,s}=Q^{n,s}+\Ps^{n,s}$,
and using \eqref{queuelength} and \eqref{19}, we obtain, for a given scenario $s$,
on the event $\Om^n_1\cap\Om^{n,s}$,
\begin{align*}
-\frac{n\epsilon}{2}&\geq X^{n,s}_N[I^{n,s}]\geq E^n_N[I^{n,s}]
-D^{n,s}_N[I^{n,s}]-1 \\
&=\sqrt{n}\hat{E}^n_N[I^{n,s}]+\lambda^n_N(\tau^{n,s}_3-\sigma^{n,s}_3)
-\sqrt{n}\hat{D}^{n,s}_N[I^{n,s}]-n\mu_N\int_{I^{n,s}}\bar{\Ps}^{n,s}_N(u)du -1.
\end{align*}
Note that
\[
\mu_N\int_{I^{n,s}}\bar\Ps^{n,s}_N(u)du\le\mu_N\rho_N(\tau^{n,s}_3-\sigma^{n,s}_3)
=\la_N(\tau^{n,s}_3-\sigma^{n,s}_3),
\]
thus
\begin{align*}
-\frac{\epsilon}{2}&\geq -2\frac{\|\hat E^n_N\|_T}{\sqrt{n}}
-2\frac{\|\hat D^n_N\|_T}{\sqrt{n}}+\frac{(\lambda_N^n-n\lambda_N)
(\tau^{n,s}_3-\sigma^{n,s}_3)}{n}-\frac{1}{n}.
\end{align*}
Since $0\le\bar\Ps^{n,s}_N\le 1$, we have $\|\hat D^{n,s}\|_T\le\|\hat S^n\|_T$.
Also, $n^{-1/2}(\la^n_N-n\la_N)$ converges. Hence
\begin{align*}
-\frac{\eps}{2}
\geq
-2\frac{\|E^n_N\|_T}{\sqrt{n}}-2\frac{\|S^n_N\|_T}{\sqrt{n}}-\frac{KT}{\sqrt n}-\frac{1}{n}\,.
\end{align*}
By the tightness of $\|\hat E^n_N\|_T$ and $\|\hat S^n_N\|_T$ for $n\in\N$ (and $T$ fixed)
and the fact that $\PP(\Om^n_1)\to1$,
we obtain $\mathbb{P}(\cup_s \Om^{n,s})\to0$.
Since $\eps>0$ is arbitrary, the result follows.
\end{proof}

Consider a stochastic differential equation with reflection, for a process $Y$ that
lives in
\[
G=\{y\in\R^N:1\cdot y\le \theta_N\},
\]
and reflects on the boundary of $G$ in the direction $-\be_N$.
Let $\{W(t)\}$ be a $(\hat\la,A)$-BM, where
$A=\diag(\lambda_1(C^2_{IA_1}+1),\ldots,\lambda_N(C^2_{IA_N}+1))$.
Let $b:\mathbb{R}^N\rightarrow \mathbb{R}^N$ be given by
\begin{equation}\label{56}
b(y)=-(\mu_1y_1,\ldots,\mu_{N-1}y_{N-1},\mu_N(y_N-(1\cdot y)^+)).
\end{equation}
Let $(X,L)$ be the unique pair of processes that is adapted to the filtration
$\sig(X_0)\vee\sig\{W(u),u\le t\}$, where $X$ has sample paths in $\mathbb{C}_G$,
$L$ has nondecreasing sample paths in $\mathbb{C}_{\R_+}$,
and the pair satisfies a.s.,
\begin{equation}\label{61}
\begin{split}
&X(t)=X_0+W(t)+\int_0^tb(X(u))du-L(t)\be_N,\qquad t\ge0,\\
&\int_{[0,\iy)}1_{\{1\cdot X(t)<\theta_N\}}dL(t)=0\,.
\end{split}
\end{equation}
The existence and uniqueness of such a pair follows from Proposition 3 of \cite{andore} on noting that
$b$ is Lipschitz. We call this pair the solution to the SDE \eqref{61}.

Define $\Gam:\mathbb{D}_{\mathbb{R}^N}([0,T])\rightarrow \mathbb{D}_{\mathbb{R}^N}([0,T])$ by
\begin{equation}\label{50}
\Gam(f)(t)=f(t)-g(t)\be_N\,,
\qquad
g(t)=\sup_{0\leq u\leq t}(\theta_N-1\cdot f(u))^-\,.
\end{equation}
The following two properties follow directly from the definition, namely
there exists a constant $C$ such that
\begin{equation}
  \label{60}
  \|\Gam(f)-\Gam(\tilde f)\|_T\le C\|f-\tilde f\|_T,\qquad f,\tilde f\in\mathbb{D}_{\R^N}([0,T]),
\end{equation}
and
\begin{equation}
  \label{62}
  w_T(\Gam(f),\cdot)\le Cw_T(f,\cdot),\qquad f\in\mathbb{D}_{\R^N}([0,T]).
\end{equation}
Given $z\in\mathbb{D}_{\R^N}$, $z(0)\in G$, we say that
$(y,\ell)\in\mathbb{D}_{\R^N}\times
\mathbb{D}_{\R}$ solves the Skorohod problem (SP)
in $G$, with reflection in the direction $-\be_N$,
for data $z$, if $y(t)\in G$ for all $t$, $\ell$ is nonnegative and nondecreasing, and
\[
y=z-\ell\be_N,\qquad\int_{[0,\iy)}1_{\{1\cdot y<\theta_N\}}d\ell=0.
\]
It is well known that for $z$ as above, a necessary and sufficient condition for $(y,\ell)$
to be a solution is that $y=\Gam(z)$ (this follows e.g., as a special case of the
much broader result of \cite{dupish1}). This will be used in the proof below.

Denote
\begin{align}\label{51}
\hat{W}_i^{n,s}(t)
=\hat{E}^n_i(t)+\frac{\lambda^n_i-n\lambda_i}{\sqrt{n}}t
-\hat{S}_i^{n}(\int_0^t\bar{\Ps}_i^{n,s}(u)du).
\end{align}
Recall conditions C1--C2 from Section \ref{sec1} that characterize $C$-tightness.
We will say that a sequence of processes $\{\xi^{n,s}\}$,
$n\in\N$, $s\in\SC_0$, with sample paths in $\mathbb{D}_{\R^k}$,
is {\it $C$-tight, uniformly in $s$} if
\begin{itemize}
\item[C1$'$.] The sequence of random variables
$\|\xi^{n,s}\|_T$ is tight for every fixed $T<\iy$, and
\item[C2$'$.] For every $T<\iy$,
$\eps>0$ and $\eta>0$ there exist $n_0$ and $\theta>0$ such that
\[
n\ge n_0 \text{ implies } \PP(\sup_s w_T(\xi^{n,s},\theta)>\eta)<\eps.
\]
\end{itemize}

\begin{prop}\label{prop1}
The sequence
$(\hat W^{n,s},\hat X^{n,s},\hat R^{n,s},\hat Q^{n,s},\hat\Ps^{n,s})$ is
$C$-tight, uniformly in $s$. Moreover,
$(\hat W^{n,0},\hat X^{n,0},\hat R^{n,0},\hat Q^{n,0},\hat\Ps^{n,0})$
converges in distribution to $(W,X,L\be_N,Q,\Ps)$,
where $(X,L)$ form the solution to the SDE \eqref{61}, and
\[
Q=(1\cdot X)^+\be_N,
\qquad
\Ps=X-Q.
\]
\end{prop}

\begin{proof}
The $C$-tightness of $\hat W^{n,s}$, uniformly in $s$,
follows from \eqref{51} using \eqref{fclt} and the fact that $\bar{\Ps}^{n,s}_i\le 1$.
By \eqref{41}--\eqref{43},
\begin{align*}
&\hat{X}^{n,s}_i=\hat{X}^n_i(0)+\hat W^{n,s}_i
-\mu_i\int_0^\cdot\hat{\Ps}^{n,s}_i(u)du-\hat{R}^{n,s}_i\,.
\end{align*}
Thus
\begin{align*}
\hat{\Ps}^{n,s}_i&=\hat{X}^n_i(0)+\hat W^{n,s}_i-\hat{Q}^{n,s}_i
-\mu_i\int_0^\cdot\hat{\Ps}^{n,s}_i(u)du-\hat{R}^{n,s}_i,\qquad i=1,\ldots,N-1,
\end{align*}
and, noting that by \eqref{23} one has $1\cdot\hat Q^{n,s}=(1\cdot\hat X^{n,s})^+$,
\begin{align*}
\hat{X}^{n,s}_N&=\hat{X}^n_N(0)+\hat{W}^{n,s}_N
-\mu_N\int_0^\cdot(\hat{X}^{n,s}_N(u)-(1\cdot\hat{X}^{n,s}(u))^+)du -\mu_N\int_0^\cdot\sum_{i=1}^{N-1}\hat{Q}^{n,s}_i(u)du-\hat{R}^{n,s}_N\\
&=\hat{X}^n_N(0)+\hat{W}^{n,s}_N-
\mu_N\int_0^\cdot(\hat{X}^{n,s}_N(u)-(\hat{X}^{n,s}_N(u)+
\sum_{i=1}^{N-1}\hat{\Ps}^{n,s}_i(u))^+)du\\ &\qquad
-\mu_N\int_0^t\sum_{i=1}^{N-1}\hat{Q}^{n,s}_i(u)du
+\mu_N\int^\cdot_0\{(1\cdot\hat{X}^{n,s}(u))^+-(\hat{X}^{n,s}_N(u)+
\sum_{i=1}^{N-1}\hat{\Ps}^{n,s}_i(u))^+\}du-\hat{R}^{n,s}_N.
\end{align*}
Defining
$Y^{n,s}_i=\hat{\Ps}^{n,s}_i$, $i=1,\ldots,N-1$, and $Y^{n,s}_N=\hat{X}^{n,s}_N$,
we have, using Lemma \ref{implem} and Lemma \ref{implem1}(i),
\begin{align}
\label{54}
Y^{n,s}_i&=\hat{X}^n_i(0)+\hat{W}^{n,s}_i-
\mu_i\int_0^\cdot Y^{n,s}_i(u)du+e^{n,s}_i,\qquad i=1,\ldots,N-1,\\
Y^{n,s}_N&=\hat{X}^n_N(0)+\hat{W}^{n,s}_N
-\mu_N\int_0^\cdot(Y^{n,s}_N(u)-(1\cdot Y^{n,s}(u))^+)du-\hat{R}^{n,s}_N+e^{n,s}\,.
\label{55}
\end{align}
Let
\begin{equation}\label{57}
F^{n,s}=1\cdot Y^{n,s}\wedge\theta_N-1\cdot Y^{n,s}\,.
\end{equation}
Then
\begin{align}\label{58}
1\cdot Y^{n,s}=\hat{X}^{n,s}_N+\sum_{i=1}^{N-1}\hat{\Ps}^{n,s}_i
=\hat{Q}^{n,s}_N+1\cdot\hat{\Ps}^{n,s}
\leq \hat{Q}^{n,s}_N\leq \theta_N +\frac{2}{\sqrt{n}}\,,
\end{align}
by \eqref{36}.
Thus $|F^{n,s}|\leq \frac{2}{\sqrt{n}}$. Further define
$\tilde{Y}^{n,s}_i=Y^{n,s}_i + \frac{1}{N}F^{n,s}$, $i=1,\ldots,N$.
Then $\tilde{Y}^{n,s}$ satisfies
\begin{equation}
  \label{52}
  \tilde Y^{n,s}(t)\in G,\qquad t\ge0,
\end{equation}
and, as follows from \eqref{56}, \eqref{54} and \eqref{55},
\begin{align}\label{53}
\tilde{Y}^{n,s}=\hat{X}^{n}(0)+\hat{W}^{n,s}
+\int_0^\cdot b(\tilde{Y}^{n,s}(u))du-\hat{R}^{n,s}_N\be_N + e^{n,s}\,.
\end{align}
Under the reference scenario, no class-$N$ reneging occurs when $\hat Q^{n,0}<\theta_N$,
that is,
\[
\int 1_{\{\hat Q^{n,0}_N(t-)<\theta_N\}}d\hat R^{n,0}_N(t)=0.
\]
As a result, the same is true with $\hat Q^{n,0}_N(t-)$ replaced by $\hat Q^{n,0}_N(t)$.
Under any other scenario, there may be one customer that does not follow the rule.
For $s=(N,j)$, $j\in\N$, write
$\tilde R^{n,s}_N$ for the normalized reneging count of all class-$N$ customers
except for customer $(N,j)$ (if it reneges).
For any other $s\in\SC_0$, let $\tilde R^{n,s}_N=\hat R^{n,s}_N$.
Then $\tilde R^{n,s}_N$ is nondecreasing and satisfies
\begin{equation}\label{59}
|\tilde R^{n,s}_N-\hat{R}^{n,s}_N|\leq n^{-1/2},
\end{equation}
and
\[
\int 1_{\{\hat{Q}^{n,s}_N(t) < \theta_N\}}d\tilde{R}^{n,s}_N(t) = 0.
\]
Let us show that $1\cdot\tilde Y^{n,s}<\theta_N$ implies $\hat Q^{n,s}_N<\theta_N$.
Indeed, by \eqref{57}, the former implies that $1\cdot Y^{n,s}<\theta_N$.
Now, $1\cdot Y^{n,s}=\hat Q^{n,s}_N+1\cdot\hat\Ps^{n,s}$, by \eqref{58}.
Thus either $\hat Q^{n,s}_N=0$, or $\hat Q^{n,s}_N>0$ in which case
$1\cdot\Ps^{n,s}=0$ by the non-idling condition \eqref{23}. In both cases,
$\hat Q^{n,s}_N<\theta_N$. It thus follows that
\begin{align}\label{sp1}
\int 1_{\{1\cdot\tilde{Y}^{n,s} < \theta_N\}}d\tilde{R}^{n,s}_N = 0.
\end{align}
By \eqref{53} and \eqref{59},
\begin{align}\label{sp2}
\tilde{Y}^{n,s}=\hat X^n(0)+\hat{W}^{n,s}+\int_0^\cdot b(\tilde{Y}^{n,s}(u))du-\tilde{R}^{n,s}\be_N
+e^{n,s}\,.
\end{align}
Hence from \eqref{52}, \eqref{sp1} and \eqref{sp2}, $(\tilde Y^{n,s},\tilde R^{n,s}_N)$
solves the aforementioned SP for the data
\[
\hat X^n(0)+\hat{W}^{n,s}+\int_0^\cdot b(\tilde{Y}^{n,s}(u))du+e^{n,s}.
\]
Therefore
\begin{align}\label{rep}
\tilde{Y}^{n,s} &= \Gam\Big(\hat X^n(0)+\hat{W}^{n,s}
+\int_0^{\cdot}b(\tilde{Y}^{n,s}(u))du+e^{n,s}\Big)\,,\\
\label{rep1}
\tilde{R}^{n,s}\be_N
&= (I - \Gam)\Big(\hat X^n(0)+\hat{W}^{n,s}+\int_0^{\cdot}b(\tilde{Y}^{n,s}(u))du+e^{n,s}\Big).
\end{align}
The convergence of $\hat X^n(0)$, the uniform $C$-tightness of $\hat W^{n,s}$, the Lipschitz property of $b$
and the Lipschitz property of $\Gam$, as expressed by \eqref{60},
imply tightness of the r.v.s $\sup_s\|\tilde Y^{n,s}\|_T$,
upon an application of Gronwall's lemma to \eqref{rep}. Hence, using again
\eqref{rep}, along with the property \eqref{62}, shows that
the processes $\tilde Y^{n,s}$ are $C$-tight, uniformly in $s$.
As a result, $\tilde R^{n,s}_N$ are also $C$-tight, uniformly in $s$.
By equations \eqref{sp2}, \eqref{rep} and \eqref{rep1}, any subsequential weak limit
of $(\hat W^{n,0},\tilde Y^{n,0},\tilde R^{n,0}_N)$ must be equal in distribution
to $(W,X,L)$. As a result, $(\hat W^{n,0},\tilde Y^{n,0},\tilde R^{n,0}_N)\To(W,X,L)$. From the definition of $\tilde{Y}^{n,s}$ and Lemma \ref{implem} it follows that
$\hat{X}^{n,s}=\tilde{Y}^{n,s}+e^{n,s}$. Moreover,
since by Lemma \ref{implem1}, $\hat R^{n,s}_i=e^{n,s}$ for $i\le N-1$,
we have $(\hat W^{n,0},\hat X^{n,0},\hat R^{n,0})\To(W,X,L\be_N)$.
Finally, the fact $\hat Q^{n,s}_i=e^{n,s}$, $i\le N-1$, stated in Lemma \ref{implem},
and the relations $1\cdot\hat Q^{n,s}=(1\cdot\hat X^{n,s})^+$,
$\hat\Ps^{n,s}=\hat X^{n,s}-\hat Q^{n,s}$ yield the result by the continuous
mapping theorem.
\end{proof}

\section{\textbf{Serve the Longest Queue}}\label{sec4}

In this section we carry out our analysis under the SLQ scheduling.
The crucial property in this case the state space collapse exhibited by the queue length processes.
Recall the constants $\theta_i$ from \eqref{35}, that determine the upper limit on the value attained
by $\hat Q^{n,s}_i$. While in the previous section the threshold of the least priority class,
$\theta_N$, was significant, under the current service policy, the property that
queue lengths remain equal makes the {\it minimal} threshold important.
Thus, assume that the classes are labeled in such a way that
\[
\theta_1\ge\cdots\ge\theta_N,
\]
and let $M=\min\{i:\theta_i=\theta_N\}$. We first treat the case $M=N$.

\begin{lem}\label{implem3}
Assume $M=N$. Fix $T$.
\\
i. For $i=1,2,\ldots,N$ we have
\begin{align*}
&\sup_s\|\hat{Q}^{n,s}_i-N^{-1}(1\cdot\hat{X}^{n,s})^+\|_T\to 0, \text{ in probability,
as } n\rightarrow \infty,\\
&\sup_s\|\bar{\Ps}^{n,s}_i-\rho_i\|_T \to 0, \text{ in probability,
as } n\rightarrow \infty\,.
\end{align*}
ii. For $i=1,2,\ldots,N-1$, $\sup_s\hat R^{n,s}_i(T)\to0$, in probability, as $n\to\iy$.
\end{lem}

\begin{proof}
Fix $\epsilon > 0$. Let $\eps_1=\frac{\eps}{4(N-1)}$ and consider the event
\begin{align*}
\Om^n = \biggl\{\hat{Q}^{n}_i(0)\leq \frac{\epsilon}{8}\,\, \mbox{and}
\,\,|\bar{\Ps}^{n}(0)-\rho_i|\leq \frac{\epsilon_1}{2}\quad \mbox{for all}
\quad i=1,\ldots,N\biggr\}\,.
\end{align*}
Then it follows from the assumptions that $\mathbb{P}({\Om}^n)\rightarrow 1$.
For $s \in\SC_0$ define
\begin{align*}
\tau^{n,s}_1 &= \inf\biggl\{t\geq 0:\min _i \hat{Q}^{n,s}_i(t)
-N^{-1}(1\cdot\hat{X}^{n,s}(t))^+ \leq -\epsilon,\\
& \hspace{8em}\text{ or } |\bar{\Ps}^{n,s}_i(t)-\rho_i| \geq \epsilon_1
\mbox{ for some } i=1,\ldots,N-1,\\
& \hspace{8em}\text{ or } |\bar{\Ps}^{n,s}_N(t)-\rho_N|
\geq \epsilon\biggr\}.
\end{align*}
Let ${A}^{n,s}=\{{\tau}^{n,s}_1\leq T\}$ and $\displaystyle{A}^n=\cup_s{A}^{n,s}$.
Now let
\begin{align*}
{A}^{n,s,i}_1&=\bigl\{\omega \in {A}^{n,s}:\hat{Q}^{n,s}_i({\tau}^{n,s}_1)
-N^{-1}(1\cdot\hat{X}^{n,s}({\tau}^{n,s}_1))^+ \leq
-\epsilon\bigr\}\cap {\Om}^n\,,\qquad i=1,\ldots,N,
\\
{A}^{n,s,i}_2&=\bigl\{\omega \in {A}^{n,s}:\min_j\hat{Q}^{n,s}_j({\tau}^{n,s}_1)
-N^{-1}(1\cdot\hat{X}^{n,s}({\tau}^{n,s}_1))^+ >
-\epsilon
\\
&\hspace{10em} \mbox{and}\,\, |\bar{\Ps}^{n,s}_i({\tau}^{n,s}_1)-\rho_i|
\geq \eps_1\bigr\}\cap {\Om}^n\,,\qquad i=1,\ldots,N-1,
\\
{A}^{n,s}_3&=\bigl\{\omega \in {A}^{n,s}:\min_j\hat{Q}^{n,s}_j({\tau}^{n,s}_1)
-N^{-1}(1\cdot\hat{X}^{n,s}({\tau}^{n,s}_1))^+ > -\epsilon,
\\
&\hspace{10em}\max_{j\le N-1} |\bar{\Ps}^{n,s}_j({\tau}^{n,s}_1)-\rho_j|< \eps_1,
\\
&\hspace{10em}\mbox{and} \,\,|\bar{\Ps}^{n,s}_N({\tau}^{n,s}_1)-\rho_N|
\geq \epsilon \bigr\}\cap {\Om}^n\,.
\end{align*}
For $\omega \in {A}^{n,s,i}_1$, there exists ${\sigma}^{n,s}_1$ such that
\begin{align}
\hat{Q}^{n,s}_i({\sigma}^{n,s}_1)-N^{-1}(1\cdot\hat{X}^{n,s}({\sigma}^{n,s}_1))^+
> -\frac{\epsilon}{2}\,\, \mbox{and, on} \,\,I^{n,s}_1:=[{\sigma}^{n,s}_1,
{\tau}^{n,s}_1],\, \hat{Q}^{n,s}_i-N^{-1}(1\cdot\hat{X}^{n,s})^+ < 0\,.
\end{align}
Note that $1\cdot\hat X^{n,s}=1\cdot\hat Q^{n,s}$, hence, on the time interval $I_1^{n,s}$,
the $i$th queue length is less than the average.
Since the scheduling policy always chooses the longest queue and on this time interval,
no customer from class $i$ enters service. Therefore the class-$i$ queue length
can only increase during this period. Thus we have
\begin{align}\label{29}
N^{-1}(1\cdot\hat{X}^{n,s})^+[I^{n,s}_1]
=N^{-1}(1\cdot\hat{Q}^{n,s})[I^{n,s}_1]
\geq \hat{Q}^{n,s}_i[I^{n,s}_1]+\frac{\epsilon}{2}\,.
\end{align}
Hence
$N^{-1}\sum_{j \neq i}\hat{Q}^{n,s}_j\geq \frac{\epsilon}{2}$, and so
by the balance equation for $Q^{n,s}$, \eqref{queuelength},
\begin{equation}\label{28}
\frac{\epsilon\sqrt nN}{2}\leq
\sum_{j\neq i}Q^{n,s}_j[I_1^{n,s}]
\leq \sum_{j \neq i}E^n_j[I_1^{n,s}]
-\sum_{j\neq i}B^{n,s}_j[I_1^{n,s}].
\end{equation}
Since, as argued above, $B^{n,s}_i[I^{n,s}]=0$,
it follows that the last term of
\eqref{28} equals $1\cdot B^{n,s}[I^{n,s}_1]$, and since $1\cdot\Ps^{n,s}=n$ on this interval,
it follows from \eqref{19} that the same term equals $1\cdot D^{n,s}[I_1^{n,s}]$.
The argument from Lemma \ref{implem} (following \eqref{26}) now shows that
$
\mathbb{P}\bigl(\cup_s{A}^{n,s,i}_1\bigr)\rightarrow 0\,.
$

Now we analyze the event ${A}^{n,s,i}_2$. By \eqref{19},
\begin{align}\notag
\bar{\Ps}^{n,s}_i(t)-\rho_i&=\bar{\Ps}^n_i(0)-\rho_i
+(\bar{B}^{n,s}_i(t)
-\lambda_it)-\frac{1}{n}\Big(S^n_i(\int_0^tn\bar{\Ps}^{n,s}_i(u)du)
-\mu_i\int_0^tn\bar{\Ps}^{n,s}_i(u)du\Big) \\
\notag
&\quad-\mu_i\int_0^t(\bar{\Ps}^{n,s}_i(u)-\rho_i)du \\
\notag
&=\bar{\Ps}^n_i(0)-\rho_i+(\bar{E}^{n}_i(t)
-\lambda_it)-\frac{1}{n}\Big(S^n_i(\int_0^tn\bar{\Ps}^{n,s}_i(u)du)
-\mu_i\int_0^tn\bar{\Ps}^{n,s}_i(u)du\Big) \\
&\quad-\mu_i\int_0^t(\bar{\Ps}^{n,s}_i(u)-\rho_i)du
-\bar{Q}^{n,s}_i(t)-\bar{R}^{n,s}_i(t)\,.
\label{73}
\end{align}
Thus, for $t\in[0,T]$,
\begin{align*}
|\bar{\Ps}^{n,s}_i(t)-\rho_i|\leq |\bar{\Ps}^n_i(0)-\rho_i|+\|\bar{E}^{n}_i
-\lambda_i\cdot\|_T+n^{-1/2}\|\hat{S}^n_i\|_T+\|\bar{Q}^{n,s}_i\|_T\\+\bar{R}^{n,s}_i(t)
+\mu_i\int_0^t|\bar{\Ps}^{n,s}_i(u)-\rho_i|du\,.
\end{align*}
And so by Gronwall's lemma we have
\begin{align*}
|\bar{\Ps}^{n,s}_i(t)-\rho_i|\leq \bigl(|\bar{\Ps}^n_i(0)
-\rho_i|+\|\bar{E}^{n}_i
-\lambda_i\cdot\|_T+n^{-1/2}\|\hat{S}^n_i\|_T+\|\bar{Q}^{n,s}_i\|_T
+\bar{R}^{n,s}_i(t)\bigr)e^{\mu_i T}.
\end{align*}
Using the identity $(1\cdot\hat{X}^{n,s})^+=1\cdot\hat{Q}^{n,s}$, we have on ${A}^{n,s,i}_2$
that $\min_j\hat Q^{n,s}\ge N^{-1}1\cdot\hat Q^{n,s}-\eps$ up to the time $\tau_1^{n,s}$.
As a result, $\max_j\hat Q^{n,s}_j\le N^{-1}1\cdot\hat Q^{n,s}+N\eps$.
Using the fact that the queue length is limited by $\hat{Q}^{n,s}_N\leq \theta_N +2n^{-1/2}$
at all times, it follows that for all large $n$,
up to time ${\tau}_1^{n,s}$,
$$
\max_j\hat{Q}^{n,s}_j\leq \theta_N +(N+1)\epsilon\,.
$$
Hence, if $\eps$ is sufficiently small then up to time ${\tau}_1^{n,s}$
there can be at most one reneging of class-$j$ customers for $j\le N-1$.
Thus, on ${A}^{n,s,i}_2$, we have
\begin{align*}
\eps_1\leq |\bar{\Ps}^{n,s}_i({\tau}_1^{n,s})-\rho_i|\leq
\bigl(|\bar{\Ps}^n_i(0)-\rho_i|+\|\bar{E}^{n}_i
-\lambda_i\cdot\|_T+n^{-1/2}\|\hat{S}^n_i\|_T+n^{-1/2}(\theta_i+1)
+n^{-1}\bigr)e^{\mu_i T}.
\end{align*}
Using the convergence of $\hat E^n$ and $\hat S^n$ \eqref{fclt} and that of
the initial condition \eqref{22}, we therefore obtain
$
\mathbb{P}\bigl(\cup_s{A}^{n,s,i}_2\bigr)\rightarrow 0\,.
$

Finally we analyze ${A}^{n,s}_3$. We have
\begin{align*}
\bar{\Ps}^{n,s}_N({\tau}_1^{n,s})&\leq 1-\sum_{i=1}^{N-1}\bar{\Ps}^{n,s}_i({\tau}_1^{n,s})
\leq \rho_N+\frac{\epsilon}{4}\,.
\end{align*}
Thus by the way $A^{n,s}_3$ is defined,
we have $\bar{\Ps}^{n,s}_N({\tau}_1^{n,s})\leq \rho_N-\epsilon$.
And so there exists ${\sigma}^{n,s}_2$ such that
\begin{align}
\bar{\Ps}^{n,s}_N({\sigma}^{n,s}_2)>\rho_N-\frac{\epsilon}{2}\,\, \mbox{and on} \,\,
[{\sigma}^{n,s}_2,{\tau}_1^{n,s}], \,\bar{\Ps}^{n,s}_N(t)<\rho_N-\frac{\epsilon}{4}\,.
\end{align}
Hence on $[{\sigma}^{n,s}_2,{\tau}_1^{n,s}]$, we have
$\sum\bar{\Ps}^{n,s}_i(t)<\sum \rho_i +\frac{\epsilon}{4}-\frac{\epsilon}{4}=1$.
Thus on this interval we have $1\cdot\hat{Q}^{n,s}=0$, and so, the argument provided in the
last part of the proof of Lemma \ref{implem1} shows
$
\mathbb{P}\bigl(\cup_s{A}^{n,s}_3\bigr)\rightarrow 0\,.
$

We have thus shown that $\mathbb{P}({A}^n)\rightarrow 0$.
The conclusion of item (i) now follows on using again the fact that
$\min_j\hat Q^{n,s}_j\ge N^{-1}1\cdot\hat Q^{n,s}-\eps$ implies
$\max_j\hat Q^{n,s}_j\le N^{-1}1\cdot\hat Q^{n,s}+N\eps$.

As for item (ii), recall that $\theta_N<\theta_i$ for all $i<M=N$.
Hence the assertion is a direct consequence of \eqref{36} and item (i).
\end{proof}

Next, consider $M\in\{1,2,\ldots,N\}$. Fix a sequence $k_n$, $n\in\N$, such that
$\lim n^{-1/2}k_n=\iy$ and $\lim n^{-1}k_n=0$. Given $T<\iy$, define
\[
T_{n,s}=\inf\{t:1\cdot R^{n,s}(t)\ge k_n\}\w T.
\]
We use the notation $U^{*, n,s}=U^{n,s}(\cdot\w T_{n,s})$ for any process $U^{n,s}$,
and refer to these processes as {\it stopped versions} of the original processes.
The following result states that Lemma \ref{implem3} is valid for the stopped processes.
\begin{lem}
  \label{lem1}
Consider general $M$.\\
i. For $i=1,2,\ldots,N$ we have
\begin{align*}
&\sup_s\|\hat{Q}^{*,n,s}_i-N^{-1}(1\cdot\hat{X}^{*,n,s})^+\|_T\to 0, \text{ in probability,
as } n\rightarrow \infty,\\
&\sup_s\|\bar{\Ps}^{*,n,s}_i-\rho_i\|_T \to 0, \text{ in probability,
as } n\rightarrow \infty\,.
\end{align*}
ii. For $i=1,2,\ldots,M-1$, $\sup_s\hat R^{*,n,s}_i(T)\to0$, in probability, as $n\to\iy$.
\end{lem}

\begin{proof}
Note that, by definition, $\bar R^{*,n,s}=e^{n,s}$. Hence a use of \eqref{73}
and again Gronwall's lemma immediately give $\bar\Ps^{*,n,s}=\rho+e^{n,s}$, proving the
second part of item (i) on the lemma.
With this at hand, the remaining assertions are proved as in Lemma \ref{implem3}.
\end{proof}

In the case where $M=N$, we provide a convergence result.
We do not attempt such an analysis for $M<N$, where, as is shown in a work
in progress \cite{ata-sah-2}, the limiting behavior may depend on properties
that are finer than first and second order data. Thus, for $M<N$,
we only obtain $C$-tightness of the processes, that however will suffice for the purpose
of proving the main result.

In order to present the result regarding the case $M=N$,
we consider an SDE of the form \eqref{61} with different domain $G$ and drift $b$.
Namely, we consider
\[
G=\{y\in\R^N:1\cdot y\le N\theta_N\},
\]
and $b:\mathbb{R}^N\rightarrow \mathbb{R}^N$ given by
\begin{equation}\label{65}
b(y)=-(\mu_1(y_1-N^{-1}(1\cdot y)^+),\ldots,\mu_N(y_N-N^{-1}(1\cdot y)^+)).
\end{equation}
The process $W(t)$ is as in Section \ref{sec3}, and the SDE of interest is now
\begin{align}\label{63}
&X(t)=X_0+W(t)+\int_0^tb(X(u))du-L(t)\be_N,\qquad t\ge0,\\ \notag
&\int_{[0,\iy)}1_{\{1\cdot X(t)<N\theta_N\}}dL(t)=0\,,
\end{align}
where a solution $(X,L)$ is defined similarly.
The map $\Gam:\mathbb{D}_{\mathbb{R}^N}([0,T]) \rightarrow \mathbb{D}_{\mathbb{R}^N}([0,T])$
that is relevant for the present setting is given by
$$
\Gam(f)(t)=f(t)-g(t)\be_N\,,
\qquad
g(t)=\sup_{0\leq u\leq t}(N\theta_N-(1\cdot f(u)))^-\,.
$$

\begin{prop}\label{prop2}
i. For general $M$, the processes
$\hat W^{n,s}$, $\hat X^{n,s}$, $\hat R^{n,s}$, $\hat Q^{n,s}$ and $\hat\Ps^{n,s}$
are $C$-tight, uniformly in $s$.

ii. In the case $M=N$, as $n\to\iy$,
$(\hat W^{n,0},\hat X^{n,0},\hat R^{n,0},\hat Q^{n,0},\hat\Ps^{n,0})$
converges in distribution to $(W,X,L\be_N,Q,\Ps)$, where $(X,L)$ form the solution to the SDE \eqref{63}, and
\[
Q=N^{-1}(1\cdot X)^+\sum_{i=1}^N\be_i,
\qquad
\Ps=X-Q.
\]
\end{prop}

\begin{proof}
{\it Step 1.}
In this and the next step we consider the case $M=N$.
We have
\begin{align*}
\hat{X}^{n,s}_i&=\hat{X}^n_i(0)+\hat{W}^{n,s}_i
-\mu_i\int_0^\cdot\hat{\Ps}^{n,s}_i(u)du-\hat{R}^{n,s}_i\\
&=\hat{X}^n_i(0)+\hat{W}^{n,s}_i-
\mu_i\int_0^\cdot(\hat{X}^{n,s}_i(u)-\hat{Q}^{n,s}_i(u))du-\hat{R}^{n,s}_i\\
&=\hat{X}^n_i(0)+\hat{W}^{n,s}_i-
\mu_i\int_0^t(\hat{X}^{n,s}_i(u)
-N^{-1}(1\cdot\hat{X}^{n,s}(u))^+)du-\hat{R}^{n,s}_i+e^{n,s}_i\,,
\end{align*}
where we have used Lemma \ref{implem3}(i)
on the last line. Next, by Lemma \ref{implem3}(ii),
\begin{align}\label{67}
\hat{X}^{n,s}=\hat X^n(0)+\hat{W}^{n,s}
+\int_0^\cdot{b}(\hat{X}^{n,s}(u))du-\hat{R}^{n,s}_N\be_N+e^{n,s}\,,
\end{align}
with $b$ as in \eqref{65}.
Define
$$
{Z}^{n,s}_i=\hat{X}^{n,s}_i+\hat{Q}^{n,s}_N-N^{-1}(1\cdot \hat{X}^{n,s})^+,
\qquad i=1,\ldots,N,
$$
and note that $Z^{n,s}=\hat X^{n,s}+e^{n,s}$. Let
$$
K^{n,s}=[N^{-1}(1\cdot{Z}^{n,s})]\wedge \theta_N -N^{-1}(1\cdot{Z}^{n,s})\,.
$$
Since
\begin{align*}
N^{-1}(1\cdot {Z}^{n,s})&=N^{-1}(1\cdot
\hat{X}^{n,s})+\hat{Q}^{n,s}_N-N^{-1}(1\cdot \hat{X}^{n,s})^+\\
&=N^{-1}(1\cdot \hat{\Ps}^{n,s})+\hat{Q}^{n,s}_N\\
&\leq \hat{Q}^{n,s}_N\leq \theta_N + 2n^{-1/2}\,,
\end{align*}
we have $K^{n,s}=e^{n,s}$. Define
$\tilde{Z}^{n,s}_i={Z}^{n,s}_i+K^{n,s}$, $i=1,2,\ldots,N$.
Then
\begin{equation}
  \label{66}
  N^{-1}(1\cdot\tilde Z^{n,s})(t)\le\theta_N,\qquad t\ge0.
\end{equation}
Moreover, $\tilde Z^{n,s}=\hat X^{n,s}+e^{n,s}$, hence
by the Lipschitz property of $b$ and \eqref{67},
\begin{align}\label{68}
\tilde{Z}^{n,s}=\hat X^n(0)+\hat{W}^{n,s}+\int_0^\cdot{b}(\tilde{Z}^{n,s}(u))du
-\hat{R}^{n,s}_N\be_N+e^{n,s}\,.
\end{align}
As in the case of FP, an argument based on the fact that
under the reference scenario no class-$i$ reneging occurs when $\hat Q^{n,0}_i<\theta_i$
shows that
\begin{equation}\label{69}
\int 1_{\{N^{-1}(1\cdot\tilde{Z}^{n,s}) < \theta_N\}}d\tilde{R}^{n,s}_N = 0,
\end{equation}
for a nonnegative, nondecreasing process $\tilde R^{n,s}_N$ that is close to $\hat R^{n,s}_N$
in the sense
\begin{equation}\label{70}
\tilde{R}^{n,s}_N=\hat{R}^{n,s}_N+e^{n,s}.
\end{equation}

{\it Step 2.}
To prove (i) (with $M=N$) and (ii),
combine \eqref{66}, \eqref{68} (with $\hat R^{n,s}_N$
replaced by $\tilde R^{n,s}_N$) and \eqref{69} to write
\begin{align}
\tilde{Z}^{n,s}&={\Gam}\Big(\hat X^n(0)+\hat{W}^{n,s}
+\int_0^{\cdot}{b}(\tilde{Z}^{n,s}(u))du+e^{n,s}\Big),\\
\tilde{R}^{n,s}\be_N &= (I - {\Gam})\Big(\hat X^n(0)+\hat{W}^{n,s}
+\int_0^{\cdot}b(\tilde{Z}^{n,s}(u))du+e^{n,s}\Big).
\end{align}
The completion of the proof, based on the above, is precisely as in Proposition \ref{prop1}.

{\it Step 3.} It remains to prove (i) for $M<N$. We start by arguing that
conclusions analogous to those obtained in Step 1 are valid here too, but for the
stopped processes. Indeed, working as in Step 1 with Lemma \ref{lem1}
in place of Lemma \ref{implem3} shows that

\begin{align*}
\hat{X}^{*,n,s}=\hat X^n(0)+\hat{W}^{*,n,s}
+\int_0^{\cdot\w T_{n,s}}{b}(\hat{X}^{n,s}(u))du-\sum_{i=M}^N\hat{R}^{*,n,s}_i\be_i+e^{n,s}\,,
\end{align*}
\begin{align}\label{68+}
\tilde{Z}^{*,n,s}=\hat X^n(0)+\hat{W}^{*,n,s}+\int_0^{\cdot\w T_{n,s}}{b}(\tilde{Z}^{n,s}(u))du
-\sum_{i=M}^N\hat{R}^{*,n,s}_i\be_i+e^{n,s}\,,
\end{align}
\begin{equation}
  \label{74}
  \tilde Z^{*,n,s}=\hat X^{*,n,s}+e^{n,s},
\end{equation}
\begin{equation}\label{69+}
\int 1_{\{N^{-1}(1\cdot\tilde{Z}^{n,s}) < \theta_N\}}d\tilde{R}^{n,s}_i = 0,
\qquad i=M,\ldots,N,
\end{equation}
for nonnegative, nondecreasing processes $\tilde R^{n,s}_i$ that are close to $\hat R^{n,s}_i$
in the sense
\begin{equation}\label{70+}
\tilde{R}^{n,s}_i=\hat{R}^{n,s}_i+e^{n,s},\qquad i=M,\ldots,N,
\end{equation}
(note that the above refers to the unstopped versions of the processes,
because again \eqref{59} is valid).

Denote
\begin{equation}\label{71}
\zeta^{n,s}=1\cdot\tilde Z^{n,s},
\qquad
\xi^{n,s}=1\cdot\hat X^{n}(0)+1\cdot\hat W^{n,s}+\int_0^\cdot1\cdot b(\tilde Z^{n,s}(u))du,
\qquad
\rho^{n,s}=\sum_{i=M}^N\tilde R^{n,s}_i.
\end{equation}
Then $\xi^{n,s}$ and $\rho^{n,s}$ have sample paths in $\mathbb{D}_{\R}$,
where those of $\rho^{n,s}$ are nonnegative and nondecreasing, and moreover,
as follows from \eqref{66}, \eqref{68+}, \eqref{69+} and \eqref{70+},
\[
\zeta^{*,n,s}=\xi^{*,n,s}+e^{n,s}-\rho^{*,n,s}\le N\theta_N,
\qquad
\int_{[0,\iy)}1_{\{\zeta^{n,s}<N\theta_N\}}d\rho^{n,s}=0.
\]
It follows that $\rho^{*,n,s}$ is given by
\begin{equation}\label{72}
\rho^{*,n,s}(t)=\sup_{0\le u\le t}(N\theta_N-\xi^{*,n,s}(u)+e^{n,s}(u))^-.
\end{equation}
We now write $c$ for generic constants and use the Lipschitz property of $b$. We have
\begin{align*}
  \rho^{*,n,s}(t) &\le c+\|\xi^{*,n,s}\|_t+e^{n,s}(t)\\
  &\le c+\|\hat X^n(0)\|+c\|\hat W^{*,n,s}\|_t
  +c\int_0^{t\w T_{n,s}}\|\tilde Z^{n,s}(u)\|du+e^{n,s}(t).
\end{align*}
Going back to \eqref{68+}
and recalling that $\rho^{n,s}$ has been defined as the sum of positive terms,
\[
\|\tilde Z^{*,n,s}(t)\|\le c\|\hat X^n(0)\|+c\|\hat W^{*,n,s}\|_t
+c\int_0^t\|\tilde Z^{*,n,s}(u)\|du+e^{n,s}(t).
\]
A use of Gronwall's lemma now shows that for $T$ fixed, $\|\tilde Z^{*,n,s}\|_T$, $n\in\N$,
are tight, uniformly in $s$.
Next, using \eqref{71} and the $C$-tightness of $\hat W^{*,n,s}$ shows that $\xi^{*,n,s}$
are $C$-tight, uniformly in $s$. In turn, using \eqref{72}, shows that so are the processes
$\rho^{*,n,s}$. In particular, for fixed $T$,
\begin{equation}
  \label{75}
  \rho^{*,n,s}(T) \text{ are tight uniformly in $s$.}
\end{equation}

Now, note that
\[
1\cdot\hat R^{*,n,s}(T)=\sum_{i=1}^{M-1}\hat R^{*,n,s}_i(T)+\rho^{*,n,s}(T) +e^{n,s}
=\rho^{*,n,s}(T) +e^{n,s},
\]
where we used Lemma \ref{lem1}(ii).
Hence in view of \eqref{75}, the definition of $T_{n,s}$, and the assumption
$\lim n^{-1/2}k_n=\iy$,
we have $\PP(\text{for some $s$, } T_{n,s}<T)\to0$ as $n\to\iy$. Thus all conclusions we have obtained
for the stopped processes are valid for the unstopped versions. Namely,
$\|\tilde Z^{n,s}\|_T$ are tight, uniformly in $s$,
$\xi^{n,s}$ and $\rho^{n,s}$ are $C$-tight uniformly in $s$, and \eqref{68+}
and \eqref{74} hold without the asterisk sign.

Using the last part of \eqref{71} and the fact that each of the processes
$\tilde R^{n,s}_i$, $n\in\N$, $i=M,\ldots,N$,
is nondecreasing shows that these processes are also $C$-tight, uniformly in $s$.
Hence by \eqref{68+}, $\tilde Z^{n,s}$, and in turn,
$\hat X^{n,s}$ are $C$-tight, uniformly in $s$.
Finally, Lemma \ref{lem1} is now valid for the processes without the asterisk sign.
Thus the uniform $C$-tightness of $\hat Q^{n,s}$ follows from that
of $\hat X^{n,s}$ upon using Lemma \ref{lem1}(i) and the continuous mapping theorem,
and that of $\hat\Ps^{n,s}$ follows from the identity \eqref{43}.
\end{proof}

\section{\textbf{Reiman's Snapshot Principle and Proof of Main Result}}\label{sec5}

We finally state and prove RSP and obtain the main result as an immediate
consequence thereof. RSP is based on the $C$-tightness of the processes
$B^{n,s}$, established as part of the limit results above.
The two policies, namely FP and SLQ, are addressed here
simultaneously.

The proof uses the following identity, that holds regardless of the service policy,
\begin{align}\label{47}
\hat{Q}^{n,s}_i(\JT_i^{n,s}(t))=\hat{B}_i^{n,s}(\JT_i^{n,s}(t)
+\WT^{n,s}_i(t))-\hat{B}^{n,s}_i(\JT_i^{n,s}(t))+\lambda_i\wh\WT^{n,s}_i(t)\,,
\end{align}
and on properties of the processes involved in it.
This identity follows from \eqref{31}, and the definition of the scaled processes,
\eqref{48} and \eqref{15}. The main argument is that the l.h.s.\ and the last term on the
r.h.s.\ must be asymptotically equal once one has that
$\hat B^{n,s}$ are uniformly $C$-tight and the term $\hat\WT^{n,s}$ is small.

\begin{prop}\label{Snapshot}
We have for $i=1,\ldots,N$,
\begin{align}\label{49}
\gamma^n_i(T):=
\sup_s\sup_{t \in [0,T]}|\hat{Q}^{n,s}_i(\JT_i^{n,s}(t))-\lambda_i\wh\WT^{n,s}_i(t)|\rightarrow 0\,\,\mbox{in probability, as}\,\,n\rightarrow \infty\,.
\end{align}
\end{prop}

\begin{proof}
First we argue that the results of Sections \ref{sec3} and \ref{sec4} imply that
$\hat B^{n,s}$ are $C$-tight, uniformly in $s$.
Indeed, by \eqref{41},
\[
\hat{B}^{n,s}_i(t)=\hat{Q}^n_i(0)+\hat{E}^n_i(t)+\hat\la_it-\hat{Q}^{n,s}_i(t)
-\hat{R}^{n,s}_i(t)+e^{n,s}(t).
\]
By \eqref{22} and \eqref{fclt}, the sum of the first two terms forms a $C$-tight sequence
of processes. By Proposition \ref{prop1}, $\hat Q^{n,s}_i$ and $\hat R^{n,s}_i$ are
$C$-tight, uniformly in $s$, under FP, and by Proposition \ref{prop2},
the same is true under SLQ. Thus follows the uniform $C$-tightness
of $\hat B^{n,s}$, and in particular, for $i=1,2,\ldots,N$ and $\eps>0$,
\begin{align}\label{64}
\lim_{\delta\downarrow 0}\limsup_{n\rightarrow \infty}
\mathbb{P}(\sup_s w_{T+2}(\hat{B}^{n,s}_i,\delta)>\epsilon)\rightarrow 0\,.
\end{align}

Fix $\eps\in(0,1)$ and define
\begin{align*}
\Om^{n,s}_i=\{\sup_{t \in [0,T]}|\JT_i^{n,s}(t)-t| > \epsilon\}\,.
\end{align*}
Then on $\Om^{n,s}_i$ there exists $t\in[0,T]$ such that
$J^{n,s}_i(t+ \epsilon)-J^{n,s}_i(t)=0$,
hence
\begin{align*}
&J^{n,s}_i(t+ \epsilon)-n\lambda_i(t+\epsilon)-[J^{n,s}_i(t)-n\lambda_i t]=-n\lambda_i \epsilon.
\end{align*}
Hence, on $\cup_s\Om^{n,s}_i$,
\begin{align*}
&\sup_s\sup_{0\leq t \leq T+1}|J^{n,s}_i(t)/n-\lambda_it|+\sup_s\sup_{0\leq t \leq T}|J^{n,s}_i(t)/n-\lambda_it|\geq\lambda_i \epsilon\,.
\end{align*}
By \eqref{40}, $J^{n,s}=E^n-R^{n,s}$, and therefore by the tightness
of $\|\hat{E}^{n}\|_{T+1}$ and $\|\hat{R}^{n,s}\|_{T+1}$, uniformly in $s$,
we have that
$$
\sup_s\sup_{0\leq t \leq T+1}\Big|\frac{J^{n,s}_i(t)-n\lambda_it}{\sqrt{n}}\Big|
$$
are tight. Hence
\begin{align}\label{imprel6}
\mathbb{P}(\sup_s\sup_{0\leq t \leq T}|\JT_i^{n,s}(t)-t| > \epsilon)\rightarrow 0,
\text{ as } n\to\iy.
\end{align}

Next we show
\begin{align}\label{imprel7}
\mathbb{P}(\sup_s\sup_{0\leq t \leq T} \WT^{n,s}_i(t)>1)\rightarrow 0\,,
\text{ as } n\to\iy.
\end{align}
For every $\omega$ in the event
under consideration there exist $t$  and $s$ such that $\WT^{n,s}_i(t)>1$. Therefore,
by \eqref{31},
\begin{align*}
Q^{n,s}_i(\JT_i^{n,s}(t))&=B_i^{n,s}(\JT_i^{n,s}(t)+\WT^{n,s}_i(t))-B^{n,s}_i(\JT_i^{n,s}(t))\\
&\geq B_i^{n,s}(\JT_i^{n,s}(t)+1)-B^{n,s}_i(\JT_i^{n,s}(t)),
\end{align*}
thus
\begin{align*}
\hat Q^{n,s}_i(\JT_i^{n,s}(t))
&\geq \hat B_i^{n,s}(\JT_i^{n,s}(t)+1)-\hat B_i^{n,s}(\JT_i^{n,s}(t))+\lambda_i\sqrt n.
\end{align*}
The conclusion follows using \eqref{imprel6} and the tightness of the r.v.s
$\sup_s\|\hat Q^{n,s}\|_{T+1}$ and $\sup_s\|\hat{B}^{n,s}\|_{T+2}$, $n\in\N$.

Using \eqref{47}, the tightness of the r.v.s $\sup_s\|\hat Q^{n,s}\|_{T+1}$
and $\sup_s\|\hat B^{n,s}\|_{T+2}$ and the facts \eqref{imprel6} and \eqref{imprel7}, gives that
of $\sup_s\|\wh\WT^{n,s}\|_T$. As a result, $\WT^{n,s}=e^{n,s}$.
Using \eqref{47} again shows that $\gamma^n_i(T)$ of \eqref{49} satisfies
\[
\gamma^n_i(T)\le \sup_s w_{T+2}(\hat B^{n,s},\del)
\]
on the event $\{\sup_s\sup_{t\le T}(\JT^{n,s}_i(t)+\WT^{n,s}_i(t))\le T+2\}\cap
\{\sup_s\WT^{n,s}_i<\del\}$. Since we have just argued that the probability of this
event converges to 1 as $n\to\iy$, the result follows from \eqref{64}.
\end{proof}

Finally we prove our main result.

\noi
\textit{Proof of Theorem \ref{th1}.}
Let $\hat{\Om}^n$ be the event defined by \eqref{30}.
Fix $(i,j)\in\SC$.
Then if
\begin{equation}\label{76}
C_{ij}^{n}(\sigma^n_{ij},\sigma^{n,ij})>C_{ij}^{n}(\bar{\sigma}^n_{ij},\sigma^{n,ij}) +\epsilon\,,
\end{equation}
we have by \eqref{14}, that $\AT^n_{ij}\leq \bar{T}$. Now, there can be two cases.

\noindent \textit{Case 1:} $h_i(\la_i^{-1}Q^{n,0}_i(\AT^n_{ij}-))< r_i$.
Then by \eqref{12}, $\Del^n_i(j)=1$, hence by \eqref{14}, $C^n_{ij}(\sig^n_{ij},\sig^{n,ij})
=h_i(\wh\WT^{n,0}_{ij})$, whereas $C^n_{ij}(\bar\sig^n_{ij},\sig^{n,ij})=r_i$. Thus
$h_i(\widehat{\WT}^{n,0}_{ij})>r_i+\epsilon$, and so
\[
h_i(\widehat{\WT}^{n,0}_{ij})-\epsilon > r_i> h_i\biggl(\frac{\hat{Q}^{n,0}_i(\AT^n_{ij}-)}{\lambda_i}\biggr).
\]
Since $\hat Q^{n,0}_i$ is bounded by $\theta_i+1$, it follows that
\begin{align}\label{45}\nonumber
\sum_{k=1}^N\frac{\gamma^n_k(T)}{\lambda_k}+\frac{1}{\sqrt{n}} &\geq \widehat{\WT}^{n,0}_{i}(\AT^n_{ij})-\frac{\hat{Q}^{n,0}_i(\JT_i^{n,s}(\AT^n_{ij}))}{\lambda_i}+\frac{1}{\sqrt{n}}= \widehat{\WT}^{n,0}_{ij}-\frac{\hat{Q}^{n,0}_i(\AT^n_{ij}-)}{\lambda_i}\\&\ge
\inf\{b-a:h(b)-h(a)>\eps,\,a\in[0,\la_i^{-1}(\theta_i+1)],b\ge0\} >0\,,
\end{align}
by the continuity of $h$.

\noindent \textit{Case 2:} $h_i(\la_i^{-1}\hat{Q}^{n,0}_i(\AT^n_{ij}-))\geq r_i$.
In this case, by \eqref{12} $\Del^n_i(j)=0$, by \eqref{14}, $C^n_{ij}(\sig^n_{ij},\sig^{n,ij})=r_i$
and $C^n_{ij}(\bar\sig^n_{ij},\sig^{n,ij})=h_i(\wh\WT^{n,s}_{ij})$. Hence
$h_i(\widehat{\WT}^{n,s}_{ij})<r_i-\epsilon$, and so
\begin{align*}
h_i(\widehat{\WT}^{n,s}_{ij})+\epsilon < r_i\leq
h_i\biggl(\frac{\hat{Q}^{n,0}_i(\AT^n_{ij}-)}{\lambda_i}\biggr).
\end{align*}
As a result,
\begin{align}\label{46}\nonumber
\sum_{k=1}^N\frac{\gamma^n_k(T)}{\lambda_k} &\geq \frac{\hat{Q}^{n,0}_i(\JT_i^{n,s}(\AT^n_{ij}))}{\lambda_i}-\frac{1}{\sqrt{n}}-\widehat{\WT}^{n,s}_{i}(\AT^n_{ij})=
\frac{\hat{Q}^{n,s}_i(\AT^n_{ij}-)}{\lambda_i}-\widehat{\WT}^{n,s}_{ij}
\\&\ge\inf\{b-a:h(b)-h(a)>\eps,\, b\in[0,\la_i^{-1}(\theta_i+1)],a\ge0\}>0\,,
\end{align}
by the continuity and strict monotonicity of $h$.

Combining \eqref{45} and \eqref{46} shows that if \eqref{76}
holds for {\it some} $(i,j)\in\SC$, then
\[
\sum_{k=1}^N\frac{\gamma^n_k(T)}{\la_k}\ge c>0,
\]
where $c$ is a constant that does not depend on $n$. Using Proposition \ref{Snapshot}
shows that $\PP((\hat{\Om}^n)^c)\to0$ as $n\to\iy$. This completes the proof.
\qed

\footnotesize

\bibliographystyle{is-abbrv}

\bibliography{refs}

\end{document}